\documentclass{amsart}

\usepackage{amssymb,amsmath}
\usepackage{verbatim}

\usepackage[section]{algorithm}
\usepackage{algpseudocode}

\newtheorem{theorem}{Theorem}[section]
\newtheorem{proposition}[theorem]{Proposition}

\newtheorem{corollary}[theorem]{Corollary}

\newtheorem{definition}[theorem]{Definition}

\newtheorem{remark}[theorem]{Remark}

\newtheorem{example}[theorem]{Example}
\newenvironment{Example}{\begin{example}\em}{\end{example}}

\def\N{{\mathbb N}}

\def\Q{{\mathbb Q}}
\def\P{{\mathbb P}}
\def\Sym{{\mathbb S}}

\def\lcm{{\mathrm{lcm}}}
\def\ord{{\mathrm{ord}}}

\def\hN{{\hat{\N}}}
\def\Inc{{\mathrm{Inc}}}

\def\NF{{\mathrm{NF}}}

\def\Ker{{\mathrm{Ker\,}}}

\def\End{{\mathrm{End}}}
\def\Aut{{\mathrm{Aut}}}

\def\Mon{{\mathrm{Mon}}}

\def\END{{\mathrm{end}}}

\def\Gr{Gr\"obner}

\def\lm{{\mathrm{lm}}}
\def\lc{{\mathrm{lc}}}
\def\lt{{\mathrm{lt}}}
\def\LM{{\mathrm{LM}}}
\def\spoly{{\mathrm{spoly}}}

\def\Reduce{{\textsc{Reduce}}}
\def\SigmaGBasis{{\textsc{SigmaGBasis}}}

\begin{document}

\title[Noetherian quotients of the algebra of partial difference $\ldots$]
{Noetherian quotients of the algebra of partial difference polynomials
and Gr\"obner bases of symmetric ideals}

\author[V. Gerdt]{Vladimir Gerdt$^*$}

\address{$^*$ Laboratory of Information Technologies, JINR, 141980 Dubna, Russia}
\email{gerdt@jinr.ru}

\author[R. La Scala]{Roberto La Scala$^{**}$}

\address{$^{**}$ Department of Mathematics, University of Bari, via Orabona 4,
70125 Bari, Italy} \email{roberto.lascala@uniba.it}

\thanks{Both authors acknowledge the support of the University of Bari
and of the visiting program of Istituto Nazionale di Alta Matematica.
The author V.G. was also supported by the grant 13-01-00668 from the Russian
Foundation for Basic Research and by grant 3802.2012.2 from the Ministry
of Education and Science of the Russian Federation.}



\maketitle

\begin{abstract}
In this paper we develop a \Gr\ bases theory for ideals of partial
difference polynomials with constant or non-constant coefficients.
In particular, we introduce a criterion providing the finiteness of
such bases when a difference ideal contains elements with suitable
linear leading monomials. This can be explained in terms of Noetherianity
of the corresponding quotient algebra. Among these Noetherian quotients
we find finitely generated polynomial algebras where the action of
suitable finite dimensional commutative algebras and in particular
finite abelian groups is defined. We obtain therefore a consistent
\Gr\ bases theory for ideals that possess such symmetries.
\end{abstract}



\section{Introduction}

The theory of difference algebras (see the books \cite{Co,KLMP,Le} and
references therein) was introduced in the 1930s by the mathematician
Joseph Fels Ritt at the same time as the theory of differential algebras.
Indeed, for a quite long time, difference algebras has attracted less
interest among researchers in comparison with differential ones
despite the fact that numerical integration of differential equations
relies on solving finite difference equations. The rapid development
of symbolic computation and computer algebra in the last decade of the
previous century gave rise to rather intensive algorithmic research in
differential algebras and to the creation of sophisticated software as the
{\em diffalg} library \cite{BH}, implementing the Rosenfeld-\Gr\ algorithm
and included in {\sc Maple} and the package {\sc LDA} \cite{GR12}.
At the same time, except for algorithmization and implementation in {\sc Maple}
of the shift algebra of linear operators \cite{Ch} as a part of the package
{\sc Ore\_algebra}, practically nothing has been developed in computer algebra
in relation to difference algebras. Nevertheless, in the last few years,
the number of applications of the theory and the methods of difference
algebras has increased fastly. For instance, it turned out that
difference \Gr\ bases may provide a very useful algorithmic tool
for the reduction of multiloop Feynman integrals in high energy physics
\cite{Ge04}, for automatic generation of finite difference approximations
to partial differential equations \cite{GBM,LM} and for the consistency
analysis of these approximations \cite{Ge12,GR10}. Relevant research
has been developed also in the context of linear functional systems
\cite{LW, Wu, ZW}. In addition to these natural applications, another
source of interest for difference algebras consists in the notion
of ``letterplace correspondence'' \cite{LSL,LSL2,LS2} which transforms
non-commutative computations for presented groups and algebras into analogue
computations with ordinary difference polynomials.
As a result of all this use, a number of computer algebra packages
implementing involutive and Buchberger's algorithms for computing
difference \Gr\ bases has been developed (see \cite{Ge12,GR12,LS} and reference
therein). A major drawback in these computations, as for the differential case,
is that such bases may be infinite owing to non-Noetherianity of the algebra
of difference polynomials. In fact, if $X$ is a finite set and $\Sigma$
denotes a multiplicative monoid isomorphic to $(\N^r,+)$ then the algebra
of difference polynomials is by definition the polynomial algebra $P$
in the infinite set of variables $X\times\Sigma$. Then, to provide
the termination of the procedures computing \Gr\ bases in $P$ at least in
some significant cases, we propose in this paper essentially two solutions.
One consists in defining an appropriate grading for $P$ that allows finite
truncated computations for difference ideals $J\subset P$ generated by
a finite number of homogeneous elements. For monomial orderings of $P$
that are compatible with such a grading this implies a criterion,
valid also for the non-graded case, which is able to certify the completeness
of a finite \Gr\ basis computed on a finite number of variables of $P$.
After the algebra of partial difference polynomials and its \Gr\ bases are
introduced in Section 2 and 3, this approach is described in Section 4
and an illustrative example based on the approximation of the Navier-Stokes
equations is given in Section 5. A second solution to the termination problem
consists in requiring that the difference ideal $J$ contains elements
with suitable linear leading monomials which corresponds to have
the Noetherian property for the quotient algebra $P/J$. Some similar
ideas appeared for the differential case in \cite{CF,Zo}. One finds this second
approach in Section 6. It is interesting to note that a relevant class
of such Noetherian quotient algebras is given by polynomial algebras $P'$
in a finite number of variables which are under the action of a tensor
product of a finite number of finite dimensional algebras generated by
single elements. These finite dimensional commutative algebras include
for instance group algebras of finite abelian groups and hence,
as a by-product of the theory of difference \Gr\ bases, one obtains
a theory for \Gr\ bases of ideals of $P'$ that are invariant under
the action of such groups or algebras (see also \cite{KLMP,St}).
These ideas are presented in Section 7 and a simple application
is described in Section 8. Finally, in Section 9 one finds conclusions
and hints for further developments of this research.


\section{Algebras of difference polynomials}

In this section we introduce the algebras of partial difference polynomials
as freely generated objects in a suitable category of commutative algebras
that are invariant under the action of a monoid isomorphic to $\N^r$
(the monoid of partial shift operators). This is a natural viewpoint
since in the formal theory of partial difference equations the unknown
functions and their shifts are assumed to be algebraically independent.
Note that one has a similar situation with the theory of algebraic equations
where the algebras of polynomials are free objects in the category
of commutative algebras.

Let $\Sigma = \langle \sigma_1,\ldots,\sigma_r\rangle$ be a free commutative
monoid which is finitely generated by the elements $\sigma_i$.
We denote $\Sigma$ in the multiplicative way with 1 as the identity element.
One has clearly that $(\Sigma,\cdot)$ is isomorphic to the additive monoid
$(\N^r,+)$ by the mapping $\sigma_1^{\alpha_1}\cdots\sigma_r^{\alpha_r}\mapsto
(\alpha_1,\ldots,\alpha_r)$. Let $K$ be a field and denote by $\End(K)$
the monoid of ring endomomorphisms of $K$. We say that {\em $\Sigma$ acts on $K$}
or equivalently that $K$ is a {\em $\Sigma$-field} if there exists a monoid
homomorphism $\rho:\Sigma\to\End(K)$. In this case, for all $\sigma\in\Sigma$
and $c\in K$ we denote $\sigma\cdot c = \rho(\sigma)(c)$. Starting from now,
we always assume that $K$ is a $\Sigma$-field. We say that $K$ is a
{\em field of constants} if $\Sigma$ acts trivially on $K$, that is,
$\sigma\cdot c = c$, for any $\sigma\in\Sigma$ and $c\in K$.

Let $A$ be a commutative $K$-algebra. We say that $A$ is a {\em $\Sigma$-algebra}
if there is a monoid homomorphism $\rho':\Sigma\to\End(A)$ extending
$\rho:\Sigma\to\End(K)$, that is, $\rho'(\sigma)(c) =
\rho(\sigma)(c)$, for all $\sigma\in\Sigma$ and $c\in K$. To simplify notations,
for any $\sigma\in\Sigma$ and $a\in A$ we put $\sigma\cdot a = \rho'(\sigma)(a)$. 
Let $B$ be a $K$-subalgebra of a $\Sigma$-algebra $A$. We call $B$
a {\em $\Sigma$-subalgebra} of $A$ if $\Sigma\cdot B = \{\sigma\cdot b\mid
\sigma\in\Sigma,b\in B\}\subset B$. In the same way, if $I$ is an ideal
of $A$ such that $\Sigma\cdot I\subset I$ then we call $I$ a {\em $\Sigma$-ideal}
of $A$. Let $B$ be a $K$-subalgebra of $A$ and let $X\subset B$ be a subset.
If $B$ is the subalgebra generated by $\Sigma\cdot X$ then $B$ coincides
clearly with the smallest $\Sigma$-subalgebra of $A$ containing $X$. In this
case, we say that $B$ is the $\Sigma$-subalgebra which is {\em $\Sigma$-generated
by $X$} and we denote it as $K[X]_\Sigma$. In a similar way,
if $X\subset I\subset A$ is the ideal generated by $\Sigma\cdot X$ then one has
that $I$ is the smallest $\Sigma$-ideal of $A$ containing $X$. Then, we say
that $I$ is the $\Sigma$-ideal which is {\em $\Sigma$-generated} by $X$ and
we make use of notation $I = \langle X \rangle_\Sigma$. We also say that $X$
is a {\em $\Sigma$-basis} of $I$. 

Let $A,B$ be $\Sigma$-algebras and let $\varphi:A\to B$ be a $K$-algebra
homomorphism. We call $\varphi$ a {\em $\Sigma$-homomorphism}
if $\varphi(\sigma\cdot a) = \sigma\cdot\varphi(a)$, for all $\sigma\in\Sigma$
and $a\in A$. In the category of $\Sigma$-algebras one can define free objects
as follows. Let $X$ be a set and denote $x(\sigma)$ each element $(x,\sigma)$
of the product set $X(\Sigma) = X\times\Sigma$. Define $P = K[X(\Sigma)]$
the $K$-algebra of polynomials in the commuting variables $x(\sigma)\in
X(\Sigma)$. For any element $\sigma\in\Sigma$, consider the ring endomorphism
$\bar{\sigma}:P\to P$ such that
\[
c x(\tau)\mapsto (\sigma\cdot c) x(\sigma\tau)
\]
for all $c\in K$ and $x(\tau)\in X(\Sigma)$. Clearly, we have a monoid
homomorphism $\rho:\Sigma\to\End(P)$ such that $\rho(\sigma) = \bar{\sigma}$,
for any $\sigma\in\Sigma$. By definition of $\bar{\sigma}$, one has that
$\rho$ extends to $P$ the action of $\Sigma$ on the base field $K$, that is,
$P$ is a $\Sigma$-algebra. Note that the homomorphism $\rho$ is in fact an
injective map. The following result states that $P$ is a free object in the
category of $\Sigma$-algebras.

\begin{proposition}
\label{freeobj}
Let $A$ be a $\Sigma$-algebra and let $f:X\to A$ be any map. Then, there exists
a unique $\Sigma$-algebra homomorphism $\varphi:P\to A$ such that
$\varphi(x(1)) = f(x)$, for all $x\in X$.
\end{proposition}

\begin{proof}
A $K$-algebra homomorphism $\varphi:P\to A$ is clearly defined by putting
$\varphi(x(\sigma)) = \sigma\cdot f(x)$, for any $x\in X$ and $\sigma\in\Sigma$.
Then, one has that $\varphi(\sigma\cdot c x(\tau)) =
\varphi((\sigma\cdot c) x(\sigma\tau)) = (\sigma\cdot c)\varphi(x(\sigma\tau)) =
(\sigma\cdot c)(\sigma\tau\cdot f(x)) = \sigma\cdot(c (\tau\cdot f(x))) =
\sigma\cdot(c \varphi(x(\tau))) = \sigma\cdot\varphi(c x(\tau))$,
for all $c\in K$, $x\in X$ and $\sigma,\tau\in\Sigma$. In other words,
the mapping $\varphi:P\to A$ is a $\Sigma$-algebra homomorphism and owing to
$x(\sigma) = \sigma\cdot x(1)$, it is clearly the unique one such that
$\varphi(x(1)) = f(x)$, for all $x\in X$.
\end{proof}

\begin{definition}
We call $P = K[X(\Sigma)]$ the {\em free $\Sigma$-algebra generated
by $X$}. In fact, $P$ is $\Sigma$-generated by the subset $X(1) =
\{x(1)\mid x\in X\}$.
\end{definition}

Note that if $A$ is any $\Sigma$-algebra which is $\Sigma$-generated by $X$
one has that $A$ is isomorphic to the quotient $P/J$ where $J\subset P$ is the
$\Sigma$-ideal containing all $\Sigma$-algebra relations satisfied by
the elements of $X$. In other words, there is a surjective $\Sigma$-algebra
homomorphism $\varphi:P\to A$ such that $x(1)\mapsto x$ ($x\in X$) and one
defines $J = \Ker\varphi$.

We are ready now to make the link with the formal theory of partial
difference equations. Let $K$ be a field of functions in the variables
$t_1,\ldots,t_r$ and fix $h_1,\ldots,h_r$ some parameters (mesh steps).
Assume we may define the action of $\Sigma$ on $K$ by putting
for all $\sigma = \prod_i \sigma_i^{\alpha_i}\in\Sigma$ and for any function
$f\in K$
\[
\sigma\cdot f(t_1,\ldots,t_r) = f(t_1 + \alpha_1 h_1,\ldots,
t_r + \alpha_r h_r)\in K.
\]
For instance, one can consider the field of rational functions
$K = F(t_1,\ldots,t_k)$ over some field $F$ and $h_1,\ldots,h_r\in F$.
Consider now a finite set of unknown functions $u_i = u_i(t_1,\ldots,t_r)$
($1\leq i\leq n$) that are assumed to be $K$-algebraically independent
together with the shifted functions $\sigma\cdot u_i = u_i(t_1 + \alpha_1 h_1,
\ldots,t_r + \alpha_r h_r)$, for any $\sigma = \prod_i \sigma_i^{\alpha_i}
\in\Sigma$. If $X = \{x_1,\ldots,x_n\}$ and if we denote $x_i(\sigma) =
\sigma\cdot u_i$ then the free $\Sigma$-algebra $P = K[X(\Sigma)]$ is
by definition the {\em algebra of partial difference polynomials}.
In particular, if $K$ is a field of constants then the difference
polynomials of $P$ are said to be {\em with constant coefficients}. Moreover,
one uses the term {\em ordinary difference} when $r = 1$. Note that
in the literature one finds the notation $P = K\{X\}$ that emphasizes
the role of $X$ as (free) $\Sigma$-generating set of the algebra $P$.
According to the notations we have introduced for the $\Sigma$-algebras
one may write also $P = K[X]_\Sigma$. In fact, we prefer $P = K[X(\Sigma)]$
to mean that $P$ is the usual polynomial algebra defined for
some special set of variables $X(\Sigma)$ which is invariant under the
action of the monoid $\Sigma$, that is, $\Sigma\cdot X(\Sigma)\subset X(\Sigma)$.
In the theory of algebraic equations we have that systems of algebraic
equations correspond to bases of ideals of the polynomial algebra.
In a similar way, one has that systems of partial difference equations
corresponds to $\Sigma$-bases of $\Sigma$-ideals of $P$ which are also
called {\em partial difference ideals}. Note that $\Sigma$ and therefore
$X(\Sigma)$ is an infinite set which implies that $P$ is not a Noetherian
algebra. Then, one has that the $\Sigma$-ideals have bases and even
$\Sigma$-bases which are generally infinite.


\section{\Gr\ bases of difference ideals}

In this section we introduce a \Gr\ basis theory for the algebra of partial
difference polynomials by extending what has be done in \cite{LS}
for the case of constant coefficients. Note that the concept of difference
\Gr\ basis has arisen also in \cite{Ge12,GR12,LSL2}.

\begin{definition}
\label{monord}
Let $\prec$ be a total ordering on the set $M = \Mon(P)$ of all monomials 
of $P$. We call $\prec$ a {\em monomial ordering of $P$} if the following
properties are satisfied:
\begin{itemize}
\item[(i)] $\prec$ is a multiplicatively compatible ordering, that is,
if $m\prec n$ then $t m \prec t n$, for any $m,n,t\in M$;
\item[(ii)] $\prec$ is a well-ordering, that is, every non-empty subset
of $M$ has a minimal element.
\end{itemize}
It is clear that in this case one has also that
\begin{itemize}
\item[(iii)] $1\prec m$, for all $m\in M, m\neq 1$.
\end{itemize}
\end{definition}

Even if the variables set $X(\Sigma)$ is infinite, by Higman's Lemma
\cite{Hi} the polynomial algebra $P = K[X(\Sigma)]$ can be always endowed
with a monomial ordering. 

\begin{proposition}
Let $\prec$ be a total ordering on $M$ which verifies the properties
$(i),(iii)$ of Definition \ref{monord}. If $\prec$ induces a well-ordering
on the variables set $X(\Sigma)\subset M$, then $\prec$ is a well-ordering
also on $M$ and hence it is a monomial ordering of $P$.
\end{proposition}

Note now that the monomials set $M$ is invariant under the action
of $\Sigma$, that is $\Sigma\cdot M\subset M$, because the same happens
to the variables set $X(\Sigma)$. Clearly, we have to require that a monomial
ordering respects this key property for defining \Gr\ bases of
$\Sigma$-ideals of $P$ which are ideals that are $\Sigma$-invariant.
In other words, one has to introduce the following notion.

\begin{definition}
Let $\prec$ be a monomial ordering of $P$. We call $\prec$ a {\em monomial
$\Sigma$-ordering of $P$} if $m\prec n$ implies that
$\sigma\cdot m\prec \sigma\cdot n$, for all $m,n\in M$ and $\sigma\in\Sigma$.
\end{definition}

Note that if $\prec$ is a monomial $\Sigma$-ordering of $P$ then one has
immediately that $\sigma\cdot m\succeq m$, for all $m\in M$ and
$\sigma\in\Sigma$. Examples of such orderings can be easily constructed
in the following way.
Let $Q = K[\sigma_1,\ldots,\sigma_r]$ be the polynomial algebra
in the variables $\sigma_j$ and therefore $\Sigma = \Mon(Q)$.
Moreover, let $K[X] = K[x_1,\ldots,x_n]$ be the polynomial algebra
in the variables $x_i$. Fix a monomial ordering $<$ for $Q$ and a monomial
ordering $\prec$ for $K[X]$. For any $\sigma\in\Sigma$, denote $X(\sigma) =
\{x_i(\sigma)\mid x_i\in X\}$. Clearly $P(\sigma) = K[X(\sigma)]$ is
a subalgebra of $P$ which is isomorphic to $K[X]$ and hence it can be
endowed with the monomial ordering $\prec$. Since $X(\Sigma) =
\bigcup_{\sigma\in\Sigma} X(\sigma)$, one can define a block monomial
ordering for $P = K[X(\Sigma)]$ obtained by $<$ and $\prec$.

\begin{proposition}
\label{weightord}
Let $m,n\in M$ be any pair of monomials. Clearly, we can factorize
these monomials as $m = m_1\cdots m_k, n = n_1\cdots n_k$ where
$m_i,$ $n_i\in M(\delta_i) = \Mon(P(\delta_i))$ $(\delta_i\in\Sigma)$
and $\delta_1 > \ldots > \delta_k$ $(k\geq 1)$. Note explicitely that
some of the factors $m_i,n_i$ may be eventually equal to 1. We define
$m\prec' n$ if and only if there is $1\leq i\leq k$ such that $m_j = n_j$ when
$j < i$ and $m_i\prec n_i$. Then, $\prec'$ is a monomial $\Sigma$-ordering
of $P$.
\end{proposition}

\begin{proof}
For all $\sigma\in\Sigma$, one has that $\sigma\cdot m =
m'_1\cdots m'_k$ where $m'_i = \sigma\cdot m_i\in M(\sigma\delta_i)$
and $\sigma\delta_1 > \ldots > \sigma\delta_k$ because $<$ is a monomial
ordering of $Q$. Assume $m\prec' n$, that is, $m_j = n_j$ for $j < i$
and $m_i\prec n_i$. Clearly, one has also that $m'_j = n'_j$.
Moreover, by definition of the monomial ordering $\prec$ on all
subalgebras $P(\sigma)\subset P$ we have that $m_i\prec n_i$ if and only if
$m'_i\prec n'_i$. We conclude that $\sigma\cdot m\prec' \sigma\cdot n$.
\end{proof}

\begin{Example}
\label{monordex}
Fix $n = 2$ and $r = 3$, that is, let $X = \{x,y\}$ and $\Sigma =
\langle \sigma_1,\sigma_2,\sigma_3 \rangle$. To simplify the notation
of the variables in $X(\Sigma)$, we identify $\Sigma$ with the additive
monoid $\N^3$, that is, we put $X(\Sigma) = \{x(i,j,k),y(i,j,k)\mid i,j,k\geq 0\}$.
By Proposition \ref{weightord}, a monomial $\Sigma$-ordering is defined
for $P = K[X(\Sigma)]$ once two monomial orderings are given
for $Q = K[\sigma_1,\sigma_2,\sigma_3]$ and $K[X] = K[x,y]$.
Consider for instance the degree reverse lexicographic
ordering $<$ on $Q$ ($\sigma_1 > \sigma_2 > \sigma_3$) and the
lexicographic ordering $\prec$ on $K[X]$ ($x\succ y$).
One has that $<$ orders the blocks of variables $X(i,j,k) =
\{x(i,j,k),y(i,j,k)\}$ in the following way
\[
\begin{array}{l}
\ldots >
\{x(2, 0, 0), y(2, 0, 0)\} > \{x(1, 1, 0), y(1, 1, 0)\} >
\{x(0, 2, 0), y(0, 2, 0)\} > \\
\hphantom{\ldots >\ }
\{x(1, 0, 1), y(1, 0, 1)\} > \{x(0, 1, 1), y(0, 1, 1)\} >
\{x(0, 0, 2), y(0, 0, 2)\} > \\
\hphantom{\ldots >\ }
\{x(1, 0, 0), y(1, 0, 0)\} > \{x(0, 1, 0), y(0, 1, 0)\} >
\{x(0, 0, 1), y(0, 0, 1)\} > \\
\hphantom{\ldots >\ }
\{x(0, 0, 0), y(0, 0, 0)\}.
\end{array}
\]
Moreover, the ordering $\prec$ is defined for each subalgebra
$K[x(i,j,k),y(i,j,k)]$. The resulting block monomial ordering for $P$
(which is a  $\Sigma$-ordering by Proposition \ref{weightord}) is therefore
the lexicographic ordering with
\[
\begin{array}{l}
\ldots \succ
x(2, 0, 0)\succ y(2, 0, 0)\succ x(1, 1, 0)\succ y(1, 1, 0)\succ
x(0, 2, 0)\succ y(0, 2, 0)\succ \\
\hphantom{\ldots \succ\ }
x(1, 0, 1)\succ y(1, 0, 1)\succ x(0, 1, 1)\succ y(0, 1, 1)\succ
x(0, 0, 2)\succ y(0, 0, 2)\succ \\
\hphantom{\ldots \succ\ }
x(1, 0, 0)\succ y(1, 0, 0)\succ x(0, 1, 0)\succ y(0, 1, 0)\succ
x(0, 0, 1)\succ y(0, 0, 1)\succ \\
\hphantom{\ldots \succ\ }
x(0, 0, 0)\succ y(0, 0, 0).
\end{array}
\]
\end{Example}

From now on, we assume that $P$ is endowed with a monomial $\Sigma$-ordering
$\prec$. Let $f = \sum_i c_i m_i\in P$ with $m_i\in M$ and $0\neq c_i\in K$.
If $m_k = \max_\prec\{m_i\}$ then we denote as usual $\lm(f) = m_k, \lc(f) = c_k$
and $\lt(f) = c_k m_k$. Since $\prec$ is a $\Sigma$-ordering, one has that
$\lm(\sigma\cdot f) = \sigma\cdot\lm(f)$ and therefore $\lc(\sigma\cdot f) =
\sigma\cdot\lc(f),\lt(\sigma\cdot f) = \sigma\cdot\lt(f)$, for all
$\sigma\in\Sigma$. If $G\subset P$ then we denote $\langle G \rangle =
\{\sum_i f_i g_i\mid f_i\in P, g_i\in G\}$, that is, $\langle G \rangle$
is the ideal of $P$ generated by $G$. Moreover, recall that
$\langle G \rangle_\Sigma = \langle \Sigma\cdot G \rangle =
\{\sum_i f_i (\delta_i\cdot g_i)\mid \delta_i\in\Sigma, f_i\in P, g_i\in G\}$
is the $\Sigma$-ideal which is $\Sigma$-generated by $G$, that is, it is
the smallest $\Sigma$-ideal of $P$ containing $G$. We call $G$ a $\Sigma$-basis
of $\langle G \rangle_\Sigma$. Finally, we put $\lm(G) =
\{\lm(f) \mid f\in G,f\neq 0\}$ and we denote $\LM(G) = \langle \lm(G) \rangle$.

\begin{proposition}
Let $G\subset P$. Then $\lm(\Sigma\cdot G) = \Sigma\cdot \lm(G)$.
In particular, if $I$ is a $\Sigma$-ideal of $P$ then $\LM(I)$
is also a $\Sigma$-ideal.
\end{proposition}

\begin{proof}
Since $P$ is endowed with a $\Sigma$-ordering, one has that
$\lm(\sigma\cdot f) = \sigma\cdot\lm(f)$, for any $f\in P,f\neq 0$ and
$\sigma\in\Sigma$. Then, $\Sigma\cdot\lm(I) = \lm(\Sigma\cdot I)\subset
\lm(I)$ and therefore $\LM(I) = \langle \lm(I) \rangle$ is a 
$\Sigma$-ideal.
\end{proof}

\begin{definition}
Let $I\subset P$ be a $\Sigma$-ideal and $G\subset I$.
We call $G$ a {\em \Gr\ $\Sigma$-basis} of $I$ if $\lm(G)$ is a
$\Sigma$-basis of $\LM(I)$. In other words, $\lm(\Sigma\cdot G) =
\Sigma\cdot\lm(G)$ is a basis of $\LM(I)$, that is, $\Sigma\cdot G$
is a \Gr\ basis of $I$ as an ideal of $P$.
\end{definition}

Since $P$ is not a Noetherian algebra, in general its $\Sigma$-ideals have
infinite (\Gr) $\Sigma$-bases. Note that one has a similar situation
for the free associative algebra and its ideals and this case is strictly
related with the algebra of ordinary difference polynomials owing to the
notion of ``letterplace correspondence'' \cite{LSL,LSL2,LS2}.
See also the comprehensive Bergman's paper \cite{Be} where the theory of
Gr\"obner bases (he did not use this name) is provided for both commutative
and non-commutative algebras in full generality, that is, without
any assumption about Noetherianity.
In Section 6 we will prove in fact the existence of a class of $\Sigma$-ideals
containing finite \Gr\ $\Sigma$-bases. According to \cite{Ge12,GR12},
such finite bases are also called ``difference \Gr\ bases''.

Let now $f,g\in P,f,g\neq 0$ and put $\lt(f) = c m, \lt(g) = d n$
with $m,n\in M$ and $c,d\in K$. If $l = \lcm(m,n)$ one defines
the {\em S-polynomial} $\spoly(f,g) = (l/c m) f - (l/d n) g$.

\begin{proposition}
\label{sigmaspoly}
For all $f,g\in P,f,g\neq 0$ and for any $\sigma\in\Sigma$ one has that
$\sigma\cdot\spoly(f,g) = \spoly(\sigma\cdot f,\sigma\cdot g)$.
\end{proposition}

\begin{proof}
Note that $\lt(\sigma\cdot f) = (\sigma\cdot c)(\sigma\cdot m),
\lt(\sigma\cdot g) = (\sigma\cdot d)(\sigma\cdot n)$ with
$\sigma\cdot m, \sigma\cdot n\in M$ and $\sigma\cdot c,\sigma\cdot d\in K$.
Since $\Sigma$ acts on the variables set $X(\Sigma)$ by injective maps,
if $l = \lcm(m,n)$ then $\sigma\cdot l = \lcm(\sigma\cdot m,\sigma\cdot n)$
and therefore we have
\begin{equation*}
\begin{gathered}
\sigma\cdot \spoly(f,g) = \sigma\cdot( \frac{l}{c m} f - \frac{l}{d n} g ) = \\
\frac{\sigma\cdot l}{(\sigma\cdot c)(\sigma\cdot m)}\sigma\cdot f -
\frac{\sigma\cdot l}{(\sigma\cdot d)(\sigma\cdot n)}\sigma\cdot g =
\spoly(\sigma\cdot f,\sigma\cdot g).
\end{gathered}
\end{equation*}
\end{proof}

In the theory of \Gr\ bases one has the following important notion.

\begin{definition}
Let $f\in P,f\neq 0$ and $G\subset P$. If $f = \sum_i f_i g_i$ with
$f_i\in P,g_i\in G$ and $\lm(f)\succeq\lm(f_i)\lm(g_i)$ for all $i$,
we say that {\em $f$ has a \Gr\ representation with respect to $G$}.
\end{definition}
 
Note that if $f = \sum_i f_i g_i$ is a \Gr\ representation then
$\sigma\cdot f = \sum_i (\sigma\cdot f_i)(\sigma\cdot g_i)$
is also a \Gr\ representation, for any $\sigma\in\Sigma$.
In fact, from $\lm(f)\succeq\lm(f_i)\lm(g_i)$ it follows that
$\lm(\sigma\cdot f) = \sigma\cdot \lm(f)\succeq
(\sigma\cdot \lm(f_i))(\sigma\cdot \lm(g_i)) = \lm(\sigma\cdot f_i)
\lm(\sigma\cdot g_i)$, for all indices $i$. Finally, if $\sigma =
\prod_i \sigma_i^{\alpha_i}, \tau = \prod_i \sigma_i^{\beta_i}
\in\Sigma = \langle \sigma_1,\ldots,\sigma_r\rangle$ we define
$\gcd(\sigma,\tau) = \prod_i \sigma_i^{\gamma_i}$ where
$\gamma_i = \min(\alpha_i,\beta_i)$. For the \Gr\ $\Sigma$-bases
of $P$ we have the following characterization.

\begin{proposition}[$\Sigma$-criterion]
\label{sigmacrit}
Let $G$ be a $\Sigma$-basis of a $\Sigma$-ideal $I\subset P$.
Then, $G$ is a \Gr\ $\Sigma$-basis of $I$ if and only if for all
$f,g\in G,f,g\neq 0$ and for any $\sigma,\tau\in\Sigma$ such that
$\gcd(\sigma,\tau) = 1$ and $\gcd(\sigma\cdot\lm(f),\tau\cdot\lm(g))\neq 1$,
the S-polynomial $\spoly(\sigma\cdot f, \tau\cdot g)$ has a \Gr\
representation with respect to $\Sigma\cdot G$.
\end{proposition}

\begin{proof}
Recall that $G$ is a \Gr\ $\Sigma$-basis if and only if $\Sigma\cdot G$
is a \Gr\ basis of $I$. By Buchberger's criterion \cite{Bu} or
by Bergman's diamond lemma \cite{Be} this happens if and only if
the S-polynomials $\spoly(\sigma\cdot f, \tau\cdot g)$ have a
\Gr\ representation with respect to $\Sigma\cdot G$, for all
$f,g\in G,f,g\neq 0$ and $\sigma,\tau\in\Sigma$. By the product criterion
(see for instance \cite{GP}) we may restrict ourselves to considering
only S-polynomials such that $\gcd(\sigma\cdot\lm(f),\tau\cdot\lm(g))\neq 1$
since $\lm(\sigma\cdot f) = \sigma\cdot\lm(f)$ and $\lm(\tau\cdot g) =
\tau\cdot\lm(g)$. Then, let $\spoly(\sigma\cdot f, \tau\cdot g)$ be any
such S-polynomial and put $\delta = \gcd(\sigma,\tau)$ and therefore
$\sigma = \delta \sigma', \tau = \delta \tau'$ with $\sigma',\tau'\in\Sigma,
\gcd(\sigma',\tau') = 1$. One has that $\spoly(\sigma\cdot f, \tau\cdot g) =
\delta\cdot \spoly(\sigma'\cdot f, \tau'\cdot g)$ owing to Proposition
\ref{sigmaspoly}. Note now that if $\spoly(\sigma'\cdot f, \tau'\cdot g) = h =
\sum_\nu f_\nu (\nu\cdot g_\nu)$ ($\nu\in\Sigma,f_\nu\in P,g_\nu\in G$)
is a \Gr\ representation with respect to $\Sigma\cdot G$ then also
$\spoly(\sigma\cdot f, \tau\cdot g) = \delta\cdot h =
\sum_\nu (\delta\cdot f_\nu) (\delta\nu\cdot g_\nu)$ is a \Gr\ representation
because $\prec$ is a $\Sigma$-ordering of $P$. We conclude that
the S-polynomials to be checked for \Gr\ representations may be restricted
to the ones satisfying both the conditions
$\gcd(\sigma\cdot\lm(f),\tau\cdot\lm(g))\neq 1$ and $\gcd(\sigma,\tau) = 1$.
\end{proof}

From the above result one obtains a variant of Buchberger's procedure
based on the ``$\Sigma$-criterion'' $\gcd(\sigma,\tau) = 1$ which is able
to compute \Gr\ $\Sigma$-bases. A standard routine that one needs in this method
is the following one.

\suppressfloats[b]
\begin{algorithm}\caption{\Reduce}
\begin{algorithmic}[0]
\State \text{Input:} $G\subset P$ and $f\in P$.
\State \text{Output:} $h\in P$ such that $f - h\in\langle G\rangle$
and $h = 0$ or $\lm(h)\notin\LM(G)$.
\State $h:= f$;
\While{ $h\neq 0$ and $\lm(h)\in\LM(G)$ }
\State choose $g\in G,g\neq 0$ such that $\lm(g)$ divides $\lm(h)$;
\State $h:= h - (\lt(h)/\lt(g)) g$;
\EndWhile;
\State \Return $h$.
\end{algorithmic}
\end{algorithm}

Note that even if $G$ may consist of an infinite number of polynomials,
the set of their leading monomials dividing $\lm(h)$ is always a finite one.
In other words, the ``choose'' instruction in the above routine can be
actually performed. Moreover, although the polynomial algebra
$P = K[X(\Sigma)]$ is infinitely generated, the existence of monomial
orderings for $P$ provides clearly the termination. By Proposition
\ref{sigmacrit} one obtains the correctness of the following procedure
for enumerating a \Gr\ $\Sigma$-basis of a $\Sigma$-ideal having
a finite $\Sigma$-basis.

\suppressfloats[b]
\floatname{algorithm}{Procedure}
\begin{algorithm}\caption{SigmaGBasis}
\begin{algorithmic}[0]
\State \text{Input:} $H$, a finite $\Sigma$-basis of a $\Sigma$-ideal
$I\subset P$.
\State \text{Output:} $G$, a \Gr\ $\Sigma$-basis of $I$.
\State $G:= \{g\in H\mid g\neq 0\}$;
\State $B:= \{(f,g) \mid f,g\in G\}$;
\While{$B\neq\emptyset$}
\State choose $(f,g)\in B$;
\State $B:= B\setminus \{(f,g)\}$;
\ForAll{$\sigma,\tau\in\Sigma$ such that $\gcd(\sigma,\tau) = 1,
\gcd(\sigma\cdot\lm(f),\tau\cdot\lm(g))\neq 1$}
\State $h:= \Reduce(\spoly(\sigma\cdot f,\tau\cdot g), \Sigma\cdot G)$;
\If{$h\neq 0$}
\State $B:= B\cup\{(g,h),(h,h) \mid g\in G\}$;
\State $G:= G\cup\{h\}$;
\EndIf;
\EndFor;
\EndWhile;
\State \Return $G$.
\end{algorithmic}
\end{algorithm}

For this procedure we do not have general termination owing to
non-Noetherianity of the algebra $P$. In fact, even if we assume that
the $\Sigma$-ideal $I\subset P$ has a finite $\Sigma$-basis,
this may be not true for its initial $\Sigma$-ideal $\LM(I)$, that is,
$I$ may have no finite \Gr\ $\Sigma$-basis. In the next section,
after introducing suitable monomial $\Sigma$-orderings of $P$ we will
give an algorithm which is able to compute in a finite number of steps
a finite \Gr\ $\Sigma$-basis whenever this exists. Note anyway that in the
above procedure all instructions can be actually performed.
In particular, for any pair of elements $f,g\in G$ and for all
$\sigma,\tau\in\Sigma$ there are only a finite number of
S-polynomials $\spoly(\sigma\cdot f,\tau\cdot g)$ satisfying both
the criteria $\gcd(\sigma,\tau) = 1$ and
$\gcd(\sigma\cdot\lm(f),\tau\cdot\lm(g))\neq 1$. A proof is given
by the arguments contained in Proposition \ref{finsigmacrit}
of the next section. Observe that the case $f = g$ has to be
considered whenever $\sigma\neq\tau$.
Finally, note that the chain criterion (see for instance \cite{GP})
can be added to \SigmaGBasis\ to shorten the number of S-polynomials
that have to be reduced. In fact, we can view this procedure as a
variant of the classical Buchberger's one applied to the basis
$\Sigma\cdot H$ of the ideal $I$ where Proposition \ref{sigmacrit}
provides the additional ``$\Sigma$-criterion'' to avoid useless pairs.
In other words, this is one way to actually implement the procedure
\SigmaGBasis\ (see \cite{LS}) in any commutative computer algebra system.

In the following sections we propose two possible solutions for
providing termination to \SigmaGBasis. First, we introduce a grading
on $P$ that is compatible with the action of $\Sigma$ which implies
that the truncated variant of this procedure with homogeneous
input stops in a finite number of steps. Another approach consists
in obtaining finite \Gr\ $\Sigma$-bases when elements with suitable
linear leading monomials belong to the given $\Sigma$-ideal $I$.
More precisely, we obtain the Noetherian property for a certain class
of (quotient) $\Sigma$-algebras $P/I$.


\section{Grading and truncation}

A useful grading for the free $\Sigma$-algebra $P$ can be introduced
in the following way. Consider the set $\hN = \N\cup\{-\infty\}$ endowed
with the binary operations $\max$ and $+$. Clearly $(\hN,\max,+)$ is
a commutative semiring which is also idempotent since $\max(d,d) = d$,
for all $d\in\hN$. Moreover, for any $\sigma =
\prod_i \sigma_i^{\alpha_i}\in\Sigma$ we put $\deg(\sigma) = \sum_i \alpha_i$.

\begin{definition}
\label{deford}
Let $\ord:M\to\hN$ be the unique mapping such that
\begin{itemize}
\item[(i)] $\ord(1) = -\infty$;
\item[(ii)] $\ord(m n) = \max(\ord(m),\ord(n))$, for all $m,n\in M$;
\item[(iii)] $\ord(x_i(\sigma)) = \deg(\sigma)$, for any variable
$x_i(\sigma)\in X(\Sigma)$.
\end{itemize}
Then, the map $\ord$ is a monoid homomorphism from $(M,\cdot)$ to
$(\hN,\max)$. We call $\ord$ the {\em order function} of $P$.
\end{definition}

More explicitely,
if $m = x_{i_1}(\delta_1)^{\alpha_1}\cdots x_{i_k}(\delta_k)^{\alpha_k}
\in M = \Mon(P)$ is any monomial different from 1 ($x_{i_l}(\delta_l)\in X(\Sigma)$
and $\alpha_l > 0$, for each $1\leq l\leq k$) we have that
\[
\ord(m) = \max(\deg(\delta_1),\ldots,\deg(\delta_k)).
\]

\begin{Example}
Let $X = \{x,y\}$ and $\Sigma = \langle \sigma_1,\sigma_2,\sigma_3 \rangle$.
As in the Example \ref{monordex}, denote $X(\Sigma) =
\{x(i,j,k),y(i,j,k)\mid i,j,k\geq 0\}$. If we consider the monomial
\[
m = y(1,1,0)^2 x(1,0,1) x(1,0,0)^3 y(0,0,0)^4
\]
then $\ord(m) = 2$.
\end{Example}

Let $P_d = \langle\, m\in M \mid \ord(m) = d\, \rangle_K\subset P$,
that is, $P_d$ is the $K$-subspace of $P$ generated by all monomials
having order equal to $d$. A polynomial $f\in P_d$ is called
{\em ord-homogeneous} and we denote $\ord(f) = d$. By property (ii)
of Definition  \ref{deford} one has clearly that $P = \bigoplus_{d\in\hN} P_d$ is
a grading of the algebra $P$ over the commutative monoid $(\hN,\max)$.

\begin{proposition}
\label{ordgood}
The following properties hold for the order function:
\begin{itemize}
\item[(i)] $\ord(\sigma\cdot m) = \deg(\sigma) + \ord(m)$, for any
$\sigma\in\Sigma$ and $m\in M$;
\item[(ii)] $\ord(\lcm(m,n)) = \ord(m n) = \max(\ord(m),\ord(n))$,
for all $m,n\in M$. Therefore, if $m\mid n$ then
$\ord(m)\leq \ord(n)$.
\end{itemize}
\end{proposition}

\begin{proof}
If $m = 1$ then $\ord(\sigma\cdot m) = \ord(m) = -\infty =
\deg(\sigma) + \ord(m)$. If otherwise $m =
x_{i_1}(\delta_1)^{\alpha_1}\cdots x_{i_k}(\delta_k)^{\alpha_k}$
then $\sigma\cdot m =
x_{i_1}(\sigma\delta_1)^{\alpha_1}\cdots x_{i_k}(\sigma\delta_k)^{\alpha_k}$
and hence $\ord(\sigma\cdot m) =
\max(\deg(\sigma\delta_1),\ldots,\deg(\sigma\delta_k)) =
\deg(\sigma) + \max(\deg(\delta_1),\ldots,\deg(\delta_k)) =
\deg(\sigma) + \ord(m)$. To prove (ii) it is sufficient to note that the order
of a monomial does not depend on the exponents of the variables occurring
in it.
\end{proof}

\begin{definition}
An ideal $I\subset P$ is called {\em $\ord$-graded} if $I = \sum_d I_d$
with $I_d = I\cap P_d$. Note that if $I$ is in addition a $\Sigma$-ideal
then by (i) of Proposition \ref{ordgood} one has that $\sigma\cdot I_d\subset
I_{\deg(\sigma) + d}$, for any $\sigma\in\Sigma$ and $d\in\hN$.
\end{definition}

Let $f,g\in P,f\neq g$ be any pair of $\ord$-homogeneous elements.
Then, the S-polynomial $h = \spoly(f,g)$ is also $\ord$-homogeneous and
by (ii) of Proposition \ref{ordgood} one has that $\ord(h) =
\max(\ord(f),\ord(g))$. If $\ord(f),\ord(g)\leq d$ for some $d\in\N$,
we have therefore that $\ord(h)\leq d$ which implies the following result.

\begin{proposition}[Termination by truncation]
\label{ordtermin}
Let $I\subset P$ be an $\ord$-graded $\Sigma$-ideal and let $d\in\N$.
Assume there is an $\ord$-homogeneous $\Sigma$-basis $H\subset I$ such that
$H_d = \{f\in H\mid \ord(f)\leq d\}$ is a finite set. Then, there exists
also an $\ord$-homogeneous \Gr\ $\Sigma$-basis $G$ of $I$ such that $G_d$
is a finite set. In other words, if one uses for \SigmaGBasis\
a selection strategy of the S-polynomials based on their orders then
the $d$-truncated variant of \SigmaGBasis\ with input $H_d$ terminates
in a finite number of steps.
\end{proposition}

\begin{proof}
In the procedure \SigmaGBasis\ one computes a subset $G$ of a \Gr\ basis
$G' = \Sigma\cdot G$ obtained by applying Buchberger's procedure to
the basis $H' = \Sigma\cdot H$ of the ideal $I$. Moreover, Proposition \ref{ordgood}
implies that the set $H'$ and hence $G'$ consists of $\ord$-homogeneous
elements. Define hence $H'_d = \{\sigma\cdot f\mid\sigma\in\Sigma,f\in H,
\deg(\sigma) + \ord(f)\leq d\}$. Note that $\Sigma_d = \{\sigma\in\Sigma\mid
\deg(\sigma)\leq d\}$ is clearly a finite set and by hypothesis we have that
$H_d$ is also a finite one. We conclude that $H'_d\subset \Sigma_d\cdot H_d$
is a finite set. Denote now by $Y_d$ the finite set of variables of $P$
occurring in the elements of $H'_d$ and define the subalgebra $P_{(d)} =
K[Y_d]\subset P$. In fact, the $d$-truncated variant of \SigmaGBasis\
computes a subset of a \Gr\ basis of the ideal $I_{(d)} \subset P_{(d)}$
generated by $H'_d$. The Noetherianity of the finitely generated polynomial
algebra $P_{(d)}$ provides then termination.
\end{proof}

Note that this result implies an algorithmic solution to the ideal membership
for finitely generated $\ord$-graded $\Sigma$-ideals. Another consequence
of the grading defined by the order function is that one has a criterion,
also in the non-graded case, for verifying that a $\Sigma$-basis computed
by the procedure \SigmaGBasis\ using a finite number of variables
of $P$ is a complete finite \Gr\ $\Sigma$-basis, whenever this basis exists.
This is of course important because actual computations can be only performed
over a finite number of variables.

\begin{definition}
Let $\prec$ be a monomial $\Sigma$-ordering of $P$. We say that $\prec$ is
{\em compatible with the order function} if $\ord(m) < \ord(n)$ implies that
$m\prec n$, for all $m,n\in M$.
\end{definition}

\begin{proposition}
Denote by $\prec$ the monomial $\Sigma$-ordering of $P$ defined in Proposition
\ref{weightord} and let $<$ be the monomial ordering of $Q =
K[\sigma_1,\ldots,\sigma_r]$ which is used to define $\prec$. Assume that $<$
is compatible with the function $\deg$, that is, $\deg(\sigma) < \deg(\tau)$
implies that $\sigma < \tau$, for any $\sigma,\tau\in\Sigma$. Then, one has
that $\prec$ is compatible with the function $\ord$.
\end{proposition}

\begin{proof}
Let $m = m_1\cdots m_k,n = n_1\cdots n_k$ be any pair of monomials of $P$,
where $m_i,n_i\in M(\delta_i)$ $(\delta_i\in\Sigma)$ and
$\delta_1 > \ldots > \delta_k$ (hence $\deg(\delta_1)\geq\ldots\geq\deg(\delta_k)$).
Assume $m\prec n$, that is, there is $1\leq i\leq k$ such that
$m_j = n_j$ when $j < i$ and $m_i\prec n_i$. If $i > 1$ or $m_i\neq 1$
one has clearly $\ord(m) = \ord(n) = \deg(\delta_1)$. Otherwise, we conclude
that $\ord(m) \leq \deg(\delta_1) = \ord(n)$.
\end{proof}

As before, we denote $\Sigma_d = \{\sigma\in\Sigma \mid \deg(\sigma)\leq d\}$.

\begin{proposition}[Finite $\Sigma$-criterion]
\label{finsigmacrit}
Assume that $P$ is endowed with a monomial $\Sigma$-ordering compatible with
the order function. Let $G\subset P$ be a finite set and define the $\Sigma$-ideal
$I = \langle G \rangle_\Sigma$.  Moreover, denote $d = \max\{\ord(\lm(g))\mid
g\in G,g\neq 0\}$. Then, $G$ is a \Gr\ $\Sigma$-basis of $I$ if and only if
for all $f,g\in G,f,g\neq 0$ and for any $\sigma,\tau\in\Sigma$ such that
$\gcd(\sigma,\tau) = 1$ and $\gcd(\sigma\cdot\lm(f),\tau\cdot\lm(g))\neq 1$,
the S-polynomial $\spoly(\sigma\cdot f,\tau\cdot g)$ has a \Gr\ representation
with respect to the finite set $\Sigma_{2d}\cdot G$.
\end{proposition}

\begin{proof}
Let $\spoly(\sigma\cdot f,\tau\cdot g) = h = \sum_\nu f_\nu (\nu\cdot g_\nu)$
be a \Gr\ representation with respect to $\Sigma\cdot G$, that is,
$\lm(h)\succeq \lm(f_\nu)(\nu\cdot \lm(g_\nu))$, for all $\nu$. We want to
bound the degree of the elements $\nu\in\Sigma$ occurring in this representation.
Put $m = \lm(f),n = \lm(g)$ and hence $\lm(\sigma\cdot f) =
\sigma\cdot m,\lm(\sigma\cdot g) = \sigma\cdot n$. By the product criterion
one has that $u = \gcd(\sigma\cdot m,\tau\cdot n)\neq 1$, that is,
there is a common variable $x_i(\sigma \alpha) = x_i(\tau \beta)$ dividing $u$
where $x_i(\alpha)$ divides $m$ and $x_i(\beta)$ divides $n$. Therefore
$\sigma \alpha = \tau \beta$ and we have that $\deg(\alpha)\leq\ord(m)\leq d$
and $\deg(\beta)\leq\ord(n)\leq d$. From $\sigma \alpha = \tau \beta$ and the
$\Sigma$-criterion $\gcd(\sigma,\tau) = 1$ it follows that
$\sigma\mid \beta,\tau\mid \alpha$ and hence $\deg(\sigma),\deg(\tau)\leq d$.
If $v = \lcm(\sigma\cdot m,\tau\cdot m)$ then we have that
$\ord(v) = \max(\deg(\sigma) + \ord(m),\deg(\tau) + \ord(n))\leq 2d$.
Clearly $v\succ \lm(h)\succeq \nu\cdot\lm(g_\nu)$ and therefore $2d\geq\ord(v)
\geq \deg(\nu) + \ord(\lm(g_\nu))\geq \deg(\nu)$.
In other words, we have that all elements $\nu$ belong to $\Sigma_{2d}$, that is,
$\spoly(\sigma\cdot f,\tau\cdot g) = \sum_\nu f_\nu (\nu\cdot g_\nu)$
is in fact a \Gr\ representation with respect to the set $\Sigma_{2d}\cdot G$.
\end{proof}

Under the assumption of a $\Sigma$-ordering compatible with the order function
and for $\Sigma$-ideals that admit finite \Gr\ $\Sigma$-bases, by the above
criterion one obtains an algorithm to compute such a basis in a finite number
of steps. In fact, this can be obtained as an adaptative procedure that keeps
the bound $2d$ for the degree of the elements of $\Sigma$ applied to the generators,
constantly updated with respect to the maximal order $d$ of the leading monomials
of the current generators. In other words, if we denote by $\SigmaGBasis(H,d)$
the variant of the procedure $\SigmaGBasis(H)$ when one substitutes $\Sigma$
with $\Sigma_d$, then we have the following algorithm.

\suppressfloats[b]
\floatname{algorithm}{Algorithm}
\begin{algorithm}\caption{SigmaGBasis2}
\begin{algorithmic}[0]
\State \text{Input:} $H$, a finite $\Sigma$-basis of a $\Sigma$-ideal
$I\subset P$ such that $\LM(I)$ has also
\State a finite $\Sigma$-basis.
\State \text{Output:} $G$, a finite \Gr\ $\Sigma$-basis of $I$.
\State $G:= \{g\in H\mid g\neq 0\}$;
\State $d':= -\infty$;
\State $d = \max\{\ord(\lm(g))\mid g\in G\}$;
\While{$d' < 2d$}
\State $d' = 2d$;
\State $G:= \SigmaGBasis(G,d')$;
\State $d = \max\{\ord(\lm(g))\mid g\in G\}$;
\EndWhile;
\State \Return $G$.
\end{algorithmic}
\end{algorithm}

Of course, the above algorithm may be refined to avoid a complete recomputation
at each step.


\section{An illustrative example}

In this section we apply the procedure \SigmaGBasis\ to an example arising
from the discretization of a well-known system of partial differential
equations. Consider the unsteady two-dimensional motion of an incompressible
viscous liquid of constant viscosity which is governed by the following system
\[
\left\{
\begin{array}{l}
\displaystyle
u_x + v_y = 0, \\
\vspace{-8pt} \\
\displaystyle
u_t + u u_x + v u_y + p_x - \frac{1}{\rho}(u_{xx} + u_{yy}) = 0, \\
\vspace{-8pt} \\
\displaystyle
v_t + u v_x + v v_y + p_y - \frac{1}{\rho}(v_{xx} + v_{yy}) = 0.
\end{array}
\right.
\]
The last two nonlinear equations are the Navier-Stokes equations and the first
linear equation is the continuity one. Equations are given in the dimensionless
form where $(u,v)$ represents the velocity field and the function $p$ is the
pressure. The parameter $\rho$ denotes the Reynolds number. For defining a finite
difference approximation of this system one has therefore to fix $X = \{u,v,p\}$
and $\Sigma = \langle \sigma_1,\sigma_2,\sigma_3 \rangle$ since all functions
are trivariate ones. To simplify the notation of the variables in $X(\Sigma)$,
we identify $\Sigma$ with the additive monoid $\N^3$ and we denote
$P = K[X(\Sigma)] = K[u(i,j,k),v(i,j,k),p(i,j,k)\mid i,j,k\geq 0]$. The base
field $K$ is the field of rational numbers. The approximation of the derivatives
of the function $u$ is given by the following formulas (forward differences)
\begin{equation*}
\begin{gathered}
u_x \approx \frac{u(x + h,y,t) - u(x,y,t)}{h} = \frac{u(1,0,0) - u(0,0,0)}{h}, \\
u_y \approx \frac{u(x,y + h,t) - u(x,y,t)}{h} = \frac{u(0,1,0) - u(0,0,0)}{h}, \\
u_t \approx \frac{u(x,y,t + h) - u(x,y,t)}{h} = \frac{u(0,0,1) - u(0,0,0)}{h}, \\
u_{xx} \approx \frac{u(x + 2h,y,t) - 2u(x + h,y,t) + u(x,y,t)}{h^2} =
\frac{u(2,0,0) - 2 u(1,0,0) + u(0,0,0)}{h^2}, \\
u_{yy} \approx \frac{u(x,y + 2h,t) - 2u(x,y + h,t) + u(x,y,t)}{h^2} =
\frac{u(0,2,0) - 2 u(0,1,0) + u(0,0,0)}{h^2} \\
\end{gathered}
\end{equation*}
where $h$ is a parameter (mesh step). One has similar approximations
for the derivatives of the functions $v,p$. If we put $H = \rho h$
then the Navier-Stokes system is approximated by the following system
of partial difference equations
\[
\left\{
\begin{array}{l}
\displaystyle
f_1 := u(1,0,0) + v(0,1,0) - u(0,0,0) - v(0,0,0)) = 0, \\
\vspace{-8pt} \\
\displaystyle
f_2 := (-u(2,0,0) -u(0,2,0) +2u(1,0,0) +2u(0,1,0) -2u(0,0,0)) \\
\vspace{-8pt} \\
\displaystyle
\quad
+\,H(p(1,0,0) +u(0,0,1) -p(0,0,0) -u(0,0,0)^2 \\
\vspace{-8pt} \\
\displaystyle
\quad
-\,(1 +v(0,0,0) -u(1,0,0)) u(0,0,0)  +u(0,1,0) v(0,0,0)) = 0, \\
\vspace{-8pt} \\
\displaystyle
f_3 := (-v(2,0,0) -v(0,2,0) +2v(1,0,0) +2v(0,1,0) -2v(0,0,0)) \\
\vspace{-8pt} \\
\displaystyle
\quad
+\,H(p(0,1,0) +v(0,0,1) -p(0,0,0) -v(0,0,0)^2 \\
\vspace{-8pt} \\
\displaystyle
\quad
+(v(1,0,0) -v(0,0,0)) u(0,0,0) -(1 -v(0,1,0)) v(0,0,0)) = 0.
\end{array}
\right.
\]
We encode this system as the $\Sigma$-ideal $I =
\langle f_1,f_2,f_3 \rangle_\Sigma\subset P$ and we want to compute a (hopefully
finite) \Gr\ $\Sigma$-basis of $I$. We may want to have such a basis
to check for the ``strong-consistency'' \cite{Ge12} of the finite difference
approximation that we are using. In fact, this property is necessary for
inheritance at the discrete level of the algebraic properties of the differential
equations. For instance, in \cite{ABGLS} we have compared the numerical behavior
of three different finite difference approximations of the Navier-Stokes equations
where just one of them is strongly consistent. The computational experiments
have confirmed the superiority of the strongly consistent approximation.
In the limit when the mesh steps go to zero, the elements in the difference
\Gr\ basis of the finite difference approximation under consideration become
differential polynomials. Then, the strong consistency holds if and only if
the latter polynomials belong to the radical differential ideal generated by
the polynomials in the input differential equations. Note that this membership
test can be done algorithmically by using the {\em diffalg} library \cite{BH}
or the differential Thomas decomposition \cite{BGLHR}. 

To perform \SigmaGBasis, we fix now the degree reverse lexicographic ordering
on the polynomial algebra $K[\sigma_1,\sigma_2,\sigma_3]$
($\sigma_1 > \sigma_2 > \sigma_3$) and the lexicographic ordering on $K[u,v,p]$
($u\succ v\succ p$). By Proposition \ref{weightord} one obtains then
a (block) monomial $\Sigma$-ordering for $P$ which is in fact the lexicographic
ordering such that
\[
\begin{array}{l}
\ldots \succ
u(2, 0, 0)\succ v(2, 0, 0)\succ p(2, 0, 0)\succ
u(1, 1, 0)\succ v(1, 1, 0)\succ p(1, 1, 0)\succ \\
\hphantom{\ldots \succ\ }
u(0, 2, 0)\succ v(0, 2, 0)\succ p(0, 2, 0)\succ
u(1, 0, 1)\succ v(1, 0, 1)\succ p(1, 0, 1)\succ \\
\hphantom{\ldots \succ\ }
u(0, 1, 1)\succ v(0, 1, 1)\succ p(0, 1, 1)\succ
u(0, 0, 2)\succ v(0, 0, 2)\succ p(0, 0, 2)\succ \\
\hphantom{\ldots \succ\ }
u(1, 0, 0)\succ v(1, 0, 0)\succ p(1, 0, 0)\succ
u(0, 1, 0)\succ v(0, 1, 0)\succ p(0, 1, 0)\succ \\
\hphantom{\ldots \succ\ }
u(0, 0, 1)\succ v(0, 0, 1)\succ p(0, 0, 1)\succ
u(0, 0, 0)\succ v(0, 0, 0)\succ p(0, 0, 0).
\end{array}
\]
Note that this ordering is compatible with the order function and hence
Proposition \ref{finsigmacrit} is applicable to certify completeness of a \Gr\
$\Sigma$-basis computed over some finite set of variables
$\{u(i,j,k),v(i,j,k),p(i,j,k)\mid i + j + k\leq d\}$.

With respect to the monomial ordering assigned to $P$, the leading monomials
of the $\Sigma$-generators of $I$ are $\lm(f_1) = u(1,0,0), \lm(f_2) = u(2,0,0),
\lm(f_3) = v(2,0,0)$. Since $\sigma_1\cdot \lm(f_1) = \lm(f_2)$, by interreducing
$f_2$ with respect to the set $\Sigma\cdot \{f_1,f_3\}$ we obtain the element
\[
\begin{array}{l}
\displaystyle
f'_2 := v(1,1,0) -u(0,2,0) -v(1,0,0) \\
\vspace{-8pt} \\
\displaystyle
\quad
+\,2u(0,1,0) -v(0,1,0) -u(0,0,0) +v(0,0,0)  \\
\vspace{-8pt} \\
\displaystyle
\quad
+\,H(p(1,0,0) +u(0,0,1) -p(0,0,0) \\
\vspace{-8pt} \\
\displaystyle
\quad
-\,(1 +v(0,1,0)) u(0,0,0) +u(0,1,0) v(0,0,0))
\end{array}
\]
whose leading monomial is $\lm(f'_2) = v(1,1,0)$. Owing to the $\Sigma$-criterion,
the only S-polynomial to consider is then
$\spoly(\sigma_1\cdot f'_2,\sigma_2\cdot f_3)$ whose reduction with respect
to $\Sigma\cdot\{f_1,f'_2,f_3\}$ leads to the new element
\[
\begin{array}{l}
\displaystyle
f_4:= p(2,0,0) +p(0,2,0) -2(p(1,0,0) + p(0,1,0) - p(0,0,0)) \\
\vspace{-8pt} \\
\displaystyle
\quad
-\,2u(0,1,0)^2 -v(0,2,0) v(1,0,0) -u(0,0,0)^2 +2v(0,0,0)^2 \\
\vspace{-8pt} \\
\displaystyle
\quad
+\,(3u(0,1,0) -2v(1,0,0) +v(0,1,0) -u(0,2,0) +v(0,0,0)) u(0,0,0) \\
\vspace{-8pt} \\
\displaystyle
\quad
-\,(3v(0,1,0) +u(0,2,0) +v(1,0,0)) v(0,0,0) \\
\vspace{-8pt} \\
\displaystyle
\quad
+\,(2v(1,0,0) -2v(0,1,0) + u(0,2,0)) u(0,1,0) \\
\vspace{-8pt} \\
\displaystyle
\quad
+\,(2v(1,0,0) + u(0,2,0) + v(0,2,0)) v(0,1,0) \\
\vspace{-8pt} \\
\displaystyle
\quad
+\,H( (u(0,1,0) +v(0,1,0)) p(0,0,0) -(u(0,1,0) +v(0,1,0)) u(0,0,1) \\
\vspace{-8pt} \\
\displaystyle
\quad
-\,p(1,0,0) v(0,1,0) -p(1,0,0) u(0,1,0) -(v(0,1,0) + 1) u(0,0,0)^2 \\
\vspace{-8pt} \\
\displaystyle
\quad
+\,(p(1,0,0) -p(0,0,0) +u(0,0,1) +v(0,1,0) \\
\vspace{-8pt} \\
\displaystyle
\quad
+(u(0,1,0)-v(0,1,0)-1) v(0,0,0) \\
\vspace{-8pt} \\
\displaystyle
\quad
+\,(v(0,1,0)+1) u(0,1,0) +v(0,1,0)^2) u(0,0,0) +u(0,1,0) v(0,0,0)^2 \\
\vspace{-8pt} \\
\displaystyle
\quad
+\,(p(1,0,0) -p(0,0,0) +u(0,0,1) -u(0,1,0) v(0,1,0) -u(0,1,0)^2) v(0,0,0)).
\end{array}
\]
The leading monomial of this difference polynomial is $\lm(f_4) = p(2,0,0)$
and no more S-polynomials have to be considered. We conclude that the set
$\{f_1,f'_2,f_3,f_4\}$ is a (finite) \Gr\ $\Sigma$-basis of the
$\Sigma$-ideal $I\subset P$. Since we make use of a monomial $\Sigma$-ordering
for $P$, this is equivalent to say that $\Sigma\cdot\{f_1,f'_2,f_3,f_4\}$
is a \Gr\ basis of the ideal $I$ and this can be verified also by applying
the classical \Gr\ bases routines to a proper truncation of the basis
$\Sigma\cdot \{f_1,f_2,f_3\}$. In fact, because the maximal order in
the input generators is 2, by Proposition \ref{finsigmacrit} it is reasonable
to bound initially the order of the variables of $P$ to 4 or 5. 
Even if it is not the case in this example, observe that the maximal order
in the elements of a \Gr\ $\Sigma$-basis may grow during the computation.
Therefore, as a general strategy, we suggest to bound the variables order
to a value which is reasonably greater than the double of the input maximal
order. The computing time for obtaining a \Gr\ basis of $I$ with
the implementation in Maple of Faug\`ere's F4 algorithm amounts to
20 seconds for order 4 and 5 hours for order 5 on our laptop Intel Core 2
Duo at 2.10 GHz with 8 GB RAM. By the procedure \SigmaGBasis\ that we
implemented in the Maple language as a variant of Buchberger's one
(see \cite{LS}), the computing time for a \Gr\ $\Sigma$-basis of $I$
is instead 0 seconds for order 4 and 3 seconds for order 5 since just two
reductions are needed. In other words, this speed-up is due to the
$\Sigma$-criterion which decreases drastically the number of S-polynomial
reductions which sometimes are very time-consuming. Note finally that the
verification method of the property of strong consistency applied to the
computed difference \Gr\ basis shows that the finite difference approximation
$\{f_1,f_2,f_3\}$ of the Navier-Stokes equations satisfies this property.


\section{A Noetherianity criterion}

As already noted, a critical feature of the algebra of partial difference
polynomials $P = K[X(\Sigma)]$ is that some of its $\Sigma$-ideals are not only
infinitely generated as ideals but also infinitely $\Sigma$-generated.
One finds an immediate counterexample for $\Sigma = \langle \sigma \rangle$,
that is, in the ordinary difference case. In fact, for some fixed variable
$x_i\in X$ one has clearly that the ideal $I = \langle x_i(1)x_i(\sigma),
x_i(1)x_i(\sigma^2),\ldots \rangle_\Sigma$ has no finite $\Sigma$-basis.
For any $x_i\in X$ and for all $\sigma^j,\sigma^k\in\Sigma$ we have that
$\sigma^k\cdot x_i(\sigma^j) = x_i(\sigma^{k+j})$ and one can identify
$\sigma^k$ with the shift map $f_k:\N\to\N$ such that $f_k(j) = k + j$
which is a strictly increasing one. It is interesting to note that
if we consider the larger monoid $\Inc(\N)$ of all strictly increasing
maps $f:\N\to\N$ acting on $P$ as $f\cdot x_i(\sigma^j) = x_i(\sigma^{f(j)})$
then one has that $P$ is $\Inc(\N)$-Noetherian \cite{AH}. In other words,
any $\Inc(\N)$-ideal of $P$ has a finite $\Inc(\N)$-basis. We may say hence
that the monoid $\Sigma$ is ``too small'' to provide $\Sigma$-Noetherianity.

One way to solve this problem is to consider suitable quotients of the algebra
of partial difference polynomials where Noetherianity and a fortiori
$\Sigma$-Noetherianity is restored. A similar approach is used for the
free associative algebra which is also non-Noetherian where the concepts
of ``algebras of solvable type, PBW algebras, G-algebras'', etc naturally
arise (see for instance \cite{Lev}).

\subsection{Countably generated algebras}
\label{countalg}

We start now with a general discussion for (commutative) algebras generated
by a countable set of elements. Let $Y = \{y_1,y_2,\ldots\}$ be a countable
set and denote $P = K[Y]$ the polynomial algebra with variables set $Y$.
Since $P$ is a free algebra, all algebras generated by a countable set of
elements are clearly isomorphic to quotients $P' = P/J$, where $J$ is some
ideal of $P$. To control the cosets in $P'$, a standard approach consists
in defining a normal form modulo $J$ associated to a monomial ordering of $P$.
Subsequently, let $\prec$ be a monomial ordering of $P$ such that
$y_1\prec y_2\prec\ldots$.

\begin{definition}
Put $M = \Mon(P)$ and denote $M'' = M\setminus\lm(J)$. Moreover,
define the $K$-subspace $P'' = \langle M'' \rangle_K\subset P$. The elements
of $M''$ are called {\em normal monomials modulo $J$ (with respect to $\prec$)}.
The polynomials in $P''$ are said to be {\em in normal form modulo $J$}.
\end{definition}

Since $P$ is endowed with a monomial ordering, by a standard argument based
on the algorithm \Reduce\ applied for the set $J$ one obtains the following
result. 

\begin{proposition}
\label{macaulay}
A $K$-linear basis of the algebra $P'$ is given by the set $M' =
\{m + J\mid m\in M''\}$.
\end{proposition}

\begin{definition}
Let $f\in P$. Denote $\NF(f)$ the unique element of $P''$ such that
$f - \NF(f)\in J$. In other words, one has $\NF(f) = \Reduce(f,J)$.
We call $\NF(f)$ the {\em normal form of $f$ modulo $J$ (with respect to $\prec$)}.
\end{definition}

By Proposition \ref{macaulay}, one has that the mapping $f + J\mapsto NF(f)$
defines a linear isomorphism between $P' = P/J$ and $P''= \langle M'' \rangle_K$.
An algebra structure is defined hence for $P''$ by imposing that such a mapping
is also an algebra isomorphism, that is, we define $f\cdot g = \NF(f g)$,
for all $f,g\in P''$. Then, we have a complete identification of $M'$ with $M''$
and $P'$ with $P''$, that is, we identify cosets with normal forms together
with their algebra structures. We will make use of this from now on.
We define hence the set of {\em normal variables}
\[
Y' = Y\cap M' = Y\setminus\lm(J).
\]
Clearly, normal variables depend strictly on the monomial ordering one uses in $P$.

\begin{proposition}[Noetherianity criterion]
\label{noethcrit}
Let $P$ be endowed with a monomial ordering. If the set of normal variables $Y'$
is finite then $P'$ is a Noetherian algebra.
\end{proposition}

\begin{proof}
It is sufficient to note that all normal monomials are products of normal
variables and therefore the quotient algebra $P' = P/J$ is in fact generated
by the set $Y'$. If $Y'$ is finite then $P'$ is a finitely generated
(commutative) algebra and hence it satisfies the Noetherian property.
\end{proof}

We need now to introduce the notion of \Gr\ basis for the ideals
of $P' = P/J$. After the identification of cosets with normal forms,
recall that $M' = M\setminus\lm(J)$ and $P' = \langle M' \rangle_K$
is a subspace of $P$ endowed with multiplication $f\cdot g = \NF(f g)$,
for all $f, g\in P'$. Then, all ideals $I'\subset P'$ have the form
$I' = I/J = \{ \NF(f)\mid f\in I\}$, for some ideal $J\subset I\subset P$.
Note that $\NF(f)\in I$ for any $f\in I$, which implies that in fact
$I' = I\cap P'$. Since the quotient algebra $P'/I'$ is isomorphic to $P/I$
and \Gr\ bases give rise to $K$-linear bases of normal monomials for the
quotients, one introduces the following definition.

\begin{definition}
\label{quogb}
Let $I' = I\cap P'$ be an ideal of $P'$ where $I$ is an ideal of $P$
containing $J$. Moreover, consider $G'\subset I'$. We call $G'$ a {\em \Gr\ basis}
of $I'$ if $G'\cup J$ is a \Gr\ basis of $I$.
\end{definition}

Let $G\subset P$. Recall that $\LM(G)$ denotes the ideal of $P$ generated
by the set $\lm(G) = \{\lm(g)\mid g\in G, g\neq 0\}$.

\begin{proposition}
\label{quogbchar}
Let $I'$ be an ideal of $P'$ and let $G'\subset I'$. Then, the set $G'$
is a \Gr\ basis of $I'$ if and only if $\LM(G') = \LM(I')$.
\end{proposition}

\begin{proof}
Let $J\subset I\subset P$ be an ideal such that $I' = I\cap P'$.
Assume $\LM(G') = \LM(I')$. Let $f\in I$ and denote $f' = \NF(f)$.
If $\lm(f)\notin\LM(J)$ then clearly $\lm(f) = \lm(f')$. Moreover,
since $\lm(f')\in\LM(I') \subset \LM(G')$ one has that $\lm(f) = \lm(f') =
m \lm(g')$, for some $m\in M, g'\in G'$. We conclude that $G'\cup J$ is a
\Gr\ basis of $I$. Suppose now that the latter condition holds. Since
$G'\subset I'$, we have clearly that $\LM(G')\subset \LM(I')$. Let now
$f'\in I'\subset I$. Then, there is $m\in M, g\in G'\cup J$ such that
$\lm(f') = m \lm(g)$. Since $\lm(f')\in M'$ then also $\lm(g)\in M'$ and
hence $g\in G'$. We conclude that $\LM(G') = \LM(I')$.
\end{proof}

\begin{proposition}
\label{normonid}
Assume that the set of normal variables $Y' = Y\cap M'$ is finite.
Then, any monomial ideal $I = \langle I\cap M' \rangle\subset P$ has a finite
basis.
\end{proposition}

\begin{proof}
It is sufficient to invoke Dickson's Lemma (see for instance \cite{CLO})
for the ideal $I$ which is generated by normal monomials that are products
of a finite number of normal variables.
\end{proof}

\begin{corollary}
\label{quofingb}
If $Y'$ is a finite set then any ideal $I'\subset P'$ has a finite
\Gr\ basis.
\end{corollary}

\begin{proof}
According to Proposition \ref{quogbchar}, consider the ideal $\LM(I')\subset P$
which is generated by the set of normal monomials $\lm(I')$. Then, it is sufficient
to apply Proposition \ref{normonid} to this ideal.
\end{proof}

It is clear that if $G$ is any \Gr\ basis of an ideal $J\neq P$ then
$Y' = Y\setminus\lm(G)$. Note that $Y$ is a countable set. Thus, if $Y'$ is finite
and hence $P' = K[Y']$ is a Noetherian algebra then $G$ needs to be an infinite
set. In general, such a \Gr\ basis cannot be computed but this may be possible
when $P'$ is a $\Sigma$-algebra owing to the notion of \Gr\ $\Sigma$-basis.

\subsection{$\Sigma$-algebras}

From now on, we assume again that $P = K[X(\Sigma)]$ is the algebra of partial
difference polynomials. Let $J\subset P$ be a $\Sigma$-ideal and define
the quotient $\Sigma$-algebra $P' = P/J$. As an algebra, we have clearly that
$P'$ is generated by the cosets $x_i(\sigma) + J$, for all $x_i(\sigma)\in X(\Sigma)$.
Moreover, $P'$ is a $\Sigma$-algebra which is $\Sigma$-generated by the cosets
$x_i(1) + J$, for any $x_i(1)\in X(1)$. In fact, $J$ is the $\Sigma$-ideal
containing all $\Sigma$-algebra relations satisfied by such generators.

Let $P$ be endowed with a monomial $\Sigma$-ordering $\prec$ and define,
as in Subsection \ref{countalg}, the set $M'\subset M = \Mon(P)$ of all
normal monomials and the set $X(\Sigma)' = X(\Sigma)\cap M'$ of all normal
variables. After the identification of cosets with normal forms,
we have that $P'$ is an algebra generated by $X(\Sigma)'$ because
normal monomials are products of normal variables.
One has also the following result.

\begin{proposition}
The $\Sigma$-algebra $P'$ is $\Sigma$-generated by $X(1)' = X(1)\cap M'$.
\end{proposition}

\begin{proof}
It is sufficient to show that $X(\Sigma)'\subset \Sigma\cdot X(1)'$.
The set of non-normal variables $X(\Sigma)\setminus X(\Sigma)' =
X(\Sigma)\cap\lm(J)$ is clearly invariant under the action of $\Sigma$.
Therefore, if $x_i(1)$ is not a normal variable then $x_i(\sigma) =
\sigma\cdot x_i(1)$ is also not a normal one. In other words,
if $x_i(\sigma)$ is a normal variable then $x_i(1)$ is also such
a variable and one has that $x_i(\sigma) = \sigma\cdot x_i(1)$.
\end{proof}

To provide the Noetherian property to the quotient algebra $P' = P/J$
by means of Proposition \ref{noethcrit} one has the following key result.

\begin{proposition}[Finiteness criterion]
\label{normfincrit}
The set of normal variables $X(\Sigma)'$ is finite if and only if
for all $1\leq i\leq n, 1\leq j\leq r$ one has that
$x_i(\sigma_j^{d_{ij}})\in\lm(J)$, for some integers $d_{ij}\geq 0$.
\end{proposition}

\begin{proof}
Put $x_i(\Sigma) = \{x_i(\sigma)\mid \sigma\in\Sigma\}$ and denote
$x_i(\Sigma)' = x_i(\Sigma)\cap X(\Sigma)'$, for any $i=1,2,\ldots,n$.
We have then to characterize when $x_i(\Sigma)'$ is a finite set.
Consider the polynomial algebra $Q = K[\sigma_1,\ldots,\sigma_r]$
and a monomial ideal $I\subset Q$. It is well-known (see for instance
\cite{CLO}, Ch.~5, \S 3, Th.~6) that the quotient algebra $Q/I$
is finite dimensional if and only if there are integers $d_j\geq 0$
such that $\sigma_j^{d_j}\in I$, for all $j=1,2,\ldots,r$. It follows that
$x_i(\Sigma)'$ is a finite set if and only if there exist integers $d_{ij}\geq 0$
such that $x_i(\sigma_j^{d_{ij}})\in\lm(J)$, for all indices $i,j$.
\end{proof}

\begin{corollary}[Termination by membership]
\label{fingb}
Let $J\subset P$ be a $\Sigma$-ideal such that for all $1\leq i\leq n,
1\leq j\leq r$ there are integers $d_{ij}\geq 0$ such that
$x_i(\sigma_j^{d_{ij}})\in\lm(J)$. Then $J$ has a finite \Gr\
$\Sigma$-basis.
\end{corollary}

\begin{proof}
Denote $I = \langle x_i(\sigma_j^{d_{ij}})\mid 1\leq i\leq n,
1\leq j\leq r \rangle_\Sigma$ and $L = \LM(J)$. Then, we have that
$I\subset L$ and the ideal $L/I\subset P/I$ has a finite basis
owing to Proposition \ref{normonid} and Proposition \ref{normfincrit}.
In other words, the $\Sigma$-ideal $L$ has a finite $\Sigma$-basis
given by the finite $\Sigma$-basis of $I$ together with the finite
basis of $L/I$.
\end{proof}

Note that the above result is not a necessary condition for finiteness
of \Gr\ $\Sigma$-bases. Consider for instance the example presented
in Section 5 of \cite{LS}. Nevertheless, Corollary \ref{fingb}
guarantees termination of the procedure \SigmaGBasis\ when a complete set
of variables $x_i(\sigma_j^{d_{ij}})$ for all $i,j$, occurs as leading monomials
of some elements of the \Gr\ $\Sigma$-basis at some intermediate
step of the computation. In other words, reaching this condition
ensures that \SigmaGBasis\ will definitely stop at some later step.
Of course, if the elements $f_{ij}\in P$ such that $\lm(f_{ij}) =
x_i(\sigma_j^{d_{ij}})$ belong to the input $\Sigma$-basis of
a $\Sigma$-ideal $J\subset P$ then we know in advance that all properties
of Noetherianity and termination are provided for the quotient $P' = P/J$.
One may have that such polynomials are themselves a \Gr\ $\Sigma$-basis of $J$
and this happens in particular in the monomial case, that is, when $J =
\langle x_i(\sigma_j^{d_{ij}})\mid 1\leq i\leq n, 1\leq j\leq r \rangle_\Sigma$,
for some $d_{ij}\geq 0$. For all $d\geq 0$, define therefore
\[
J^{(d)} = \langle x_i(\sigma)\mid 1\leq i\leq n,
\deg(\sigma) = d + 1 \rangle_\Sigma \supset
\langle x_i(\sigma_j^{d+1})\mid 1\leq i\leq n,
1\leq j\leq r \rangle_\Sigma
\]
and put $J^{(-\infty)} = \langle X(1) \rangle_\Sigma = \langle X(\Sigma) \rangle$.
If $P = \bigoplus_{d\in\hN} P_d$ is the grading of $P$ defined by the order
function then the subalgebra $P^{(d)} = \bigoplus_{i\leq d} P_i\subset P$
is clearly isomorphic to the quotient $P/J^{(d)}$ and hence it can be endowed
with the structure of a $\Sigma$-algebra. Then, to make use of the following
filtration of subalgebras
\[
K = P^{(-\infty)}\subset P^{(0)}\subset P^{(1)}\subset \ldots \subset P
\]
to perform concrete computations with \Gr\ $\Sigma$-bases as explained
in Section 4 corresponds to work progressively modulo the $\Sigma$-ideals
\[
\langle X(\Sigma) \rangle = J^{(-\infty)}\supset J^{(0)}\supset J^{(1)}\supset
\ldots\supset 0
\]
providing the finite set of normal variables $X(\Sigma_d) = \{x_i(\sigma)\mid
1\leq i\leq n,\deg(\sigma)\leq d\}$ and hence the Noetherian property
for each quotient $P/J^{(d)}$ isomorphic to $P^{(d)}$. In other words,
termination by truncation is essentially a special instance of termination
by membership. Another interesting case is the ordinary one, that is,
when $\Sigma = \langle \sigma \rangle$. In this case, any set of polynomials
$f_1,\ldots,f_n\in P$ such that $\lm(f_i) = x_i(\sigma^{d_i})$ ($d_i\geq 0$)
is a \Gr\ $\Sigma$-basis since all S-polynomials trivially reduce to zero
according to the product criterion.

To motivate the last result of this section, let us consider the following
problem. Assume that $K$ is a field of constants and let $V$ be a finite
dimensional $K$-vector space. Denote by $\END_K(V)$ the algebra of
$K$-linear endomorphisms of $V$ and let $Q'\subset\END_K(V)$ be a subalgebra
generated by $r$ commuting endomorphisms. Since $Q = K[\sigma_1,\ldots,\sigma_r]$
is the free commutative algebra with $r$ generators, one has a $K$-algebra
homomorphism $Q\to \END_K(V)$ sending the $\sigma_i$ onto the generators
of $Q'$, that is, $V$ is a $Q$-module. Consider now
the (Noetherian) polynomial algebra $R$ whose variables are a $K$-linear
basis of $V$. In other words, $V$ is the subspace of linear forms of $R$
or equivalently $R$ is the symmetric algebra on $V$.
Define $\End_K(R)$ the monoid of $K$-algebra endomorphisms of $R$.
Since $\Sigma = \Mon(Q)$, we can extend the action of $\Sigma$ on $V$
to a monoid homomorphism $\Sigma\to \End_K(R)$, that is, $R$ is
a $\Sigma$-algebra. Because $P = K[X(\Sigma)]$ is a free $\Sigma$-algebra,
there is a suitable set $X = \{x_1,\ldots,x_n\}$ and a $\Sigma$-ideal
$J\subset P$ such that $R$ is isomorphic to the quotient $\Sigma$-algebra
$P' = P/J$. Since $Q$ acts linearly over $V$, one has that $J$ is
$\Sigma$-generated by linear polynomials. Then, in the following result
we analyze from the perspective of Proposition \ref{noethcrit} and
Proposition \ref{normfincrit} the easiest case for a linear $\Sigma$-ideal
providing the Noetherian property to the quotient $\Sigma$-algebra.
In Section 7 we will show that this case corresponds to have the finite
dimensional commutative algebra $Q'$ decomposable as the tensor product
of $r$ cyclic subalgebras. This happens in particular if $Q'$ is the group
algebra of a finite abelian group and one application of this specific case
is given in Section 8.

\begin{proposition}
\label{finlinact}
Let $K$ be a field of constants and consider the linear polynomials
$f_{ij} = \sum_{0\leq k\leq d_{ij}} c_{ijk} x_i(\sigma_j^k)\in P$
where $c_{ijk}\in K$ and $c_{ijd_{ij}} = 1$, for all $1\leq i\leq n,
1\leq j\leq r$. Then $\lm(f_{ij}) = x_i(\sigma_j^{d_{ij}})$ and the set
$\{f_{ij}\}$ is a \Gr\ $\Sigma$-basis.
\end{proposition}

\begin{proof}
Since $X(\Sigma)$ is endowed with a $\Sigma$-ordering, one has that
$x_i(\sigma_j^k)\prec x_i(\sigma_j^l)$ if $k < l$ and hence
$\lm(f_{ij}) = x_i(\sigma_j^{d_{ij}})$. Then, the only S-polynomials
to be considered are
\[
s = \spoly(\sigma_q^{d_{iq}}\cdot f_{ip}, \sigma_p^{d_{ip}}\cdot f_{iq}) =
\sum_{0\leq k<d_{ip}} c_{ipk} x_i(\sigma_q^{d_{iq}}\sigma_p^k) -
\sum_{0\leq l<d_{iq}} c_{iql} x_i(\sigma_p^{d_{ip}}\sigma_q^l),
\]
for all $1\leq i\leq n$ and $1\leq p<q\leq r$. By reducing $s$ with
polynomials $\sigma_p^k\cdot f_{iq}$ and $\sigma_q^l\cdot f_{ip}$
one obtains
\[
s' = -
\sum_{0\leq k<d_{ip},0\leq l<d_{iq}} c_{ipk} c_{iql} x_i(\sigma_q^l\sigma_p^k) +
\sum_{0\leq l<d_{iq},0\leq k<d_{ip}} c_{iql} c_{ipk} x_i(\sigma_p^k\sigma_q^l) = 0.
\]
\end{proof}

Note explicitely that the assumption that $K$ is a field of constants is necessary
in the above result. In fact, if $\Sigma$ acts on $K$ in a non-trivial way then
generally
\begin{equation*}
\begin{gathered}
s' = -
\sum_{0\leq k<d_{ip},0\leq l<d_{iq}} (\sigma_q^{d_{iq}}\cdot c_{ipk})
(\sigma_p^k\cdot c_{iql}) x_i(\sigma_q^l\sigma_p^k) \\
\qquad \,+ \sum_{0\leq l<d_{iq},0\leq k<d_{ip}} (\sigma_p^{d_{ip}}\cdot c_{iql})
(\sigma_q^l\cdot c_{ipk}) x_i(\sigma_p^k\sigma_q^l)\neq 0.
\end{gathered}
\end{equation*}


\section{A Noetherian $\Sigma$-algebra of special interest}

From now on we assume that $K$ is a field of constants.
We define the ideal $J = \langle f_{ij} \rangle_\Sigma\subset P$ where
$f_{ij} =\sum_{0\leq k\leq d_{ij}} c_{ijk} x_i(\sigma_j^k)$ ($c_{ijk}\in K,
c_{ijd_{ij}} = 1$), for any $1\leq i\leq n,1\leq j\leq r$. We want
to describe the (Noetherian) $\Sigma$-algebra $P' = P/J$.
To simplify notations and since they are interesting in themselves,
we consider separately the cases when $r = 1$ and $n = 1$.

First assume that $r = 1$, that is, $\Sigma = \langle \sigma \rangle$ and hence
$P' = P/J$ where $J = \langle f_1,\ldots,f_n \rangle_\Sigma$ with $f_i =
\sum_{0\leq k\leq d_i} c_{ik} x_i(\sigma^k)$ ($c_{ik}\in K, c_{id_i} = 1$).
Define $Q = K[\sigma]$ the algebra of polynomials in the single variable
$\sigma$ and denote $g_i = \sum_{0\leq k\leq d_i} c_{ik} \sigma^k\in Q$.
Moreover, put $d = \sum_i d_i$ and let $V = K^d$. Finally, consider the
$d\times d$ block-diagonal matrix
\[
A = A_1 \oplus\ldots\oplus A_n =
\left(
\begin{array}{cccc}
A_1 &  0  & \ldots & 0 \\
 0  & A_2 & \ldots & 0 \\
\vdots & \vdots & & \vdots \\
 0  &  0  & \ldots & A_n \\
\end{array}
\right)
\]
where each block $A_i$ is the companion matrix of the polynomial $g_i$,
that is,
\[
A_i =
\left(
\begin{array}{cccccc}
0 & 0 & \ldots & 0 & - c_{i0} \\
1 & 0 & \ldots & 0 & - c_{i1} \\
0 & 1 & \ldots & 0 & - c_{i2} \\
\vdots & \vdots & \ddots & \vdots & \vdots \\
0 & 0 & \ldots & 1 & - c_{id-1} \\
\end{array}
\right).
\]
Note that $A$ has all entries in the base field $K$ and it can be considered
as the Frobenius normal form of a $d\times d$ matrix provided that
$g_1\mid\ldots\mid g_n$. Recall that any square matrix is similar
over the base field to its Frobenius normal form, that is, we are considering
any $K$-linear endomorphism of $V$.
Then, the monoid $\Sigma$ or equivalently the algebra $Q$ acts linearly
over the vector space $V$ by means of the representation $\sigma^k\mapsto A^k$.
If $\{v_q\}_{1\leq q\leq d}$ is the canonical basis of $V$, we denote 
$x_i(\sigma^k) = v_q$ where $q = \sum_{j < i} d_j + k + 1$ for all $1\leq i\leq n,
0\leq k < d_i$. We have hence $x_i(\sigma^k) = A^k x_i(1) = \sigma^k\cdot x_i(1)$.
In other words, for the $Q$-module $V$ one has the decomposition
$V = \bigoplus_i V_i$ where $V_i$ is the cyclic submodule generated
by $x_i(1)$ and annihilated by the ideal $\langle g_i \rangle\subset Q$.
Denote now by $R$ the (Noetherian) polynomial algebra generated by the
finite set of variables $X(\Sigma)' = \{x_i(\sigma^k)\mid 1\leq i\leq n,
0\leq k < d_i\}$, that is, $V$ coincides with the subspace of linear forms
of $R$. Then, one extends the action of the monoid $\Sigma =
\langle \sigma \rangle$ to the polynomial algebra $R$ in the natural way
that is by putting, for all $k\geq 0$ and $x_i(\sigma^j)\in X(\Sigma)'$
\[
\sigma^k\cdot x_i(\sigma^j) = A^k x_i(\sigma^j).
\]

Denote by $\END_K(P)$ the algebra of all $K$-linear mappings $P\to P$
and define by $\End_K(P)$ the monoid of $K$-algebra endomorphisms of $P$.
Note that the representation $\rho:\Sigma\to\End_K(P)$ can be extended
linearly to $\bar{\rho}:Q\to\END_K(P)$. Then, one has that $f_i =
\sum_k c_{ik} x_i(\sigma^k) = \sum_k c_{ik} \sigma^k\cdot x_i(1) =
g_i\cdot x_i(1)$, for all $i = 1,2,\ldots,n$. 

\begin{proposition}
\label{iso1}
If $\Sigma = \langle \sigma \rangle$ then the $\Sigma$-algebras $P',R$
are $\Sigma$-isomorphic.
\end{proposition}

\begin{proof}
By Proposition \ref{finlinact} we have that the set $\{f_i\}$ is a \Gr\
$\Sigma$-basis of the $\Sigma$-ideal $J\subset P$ and it is clear that
the set of normal variables modulo $J$ is exactly $X(\Sigma)' =
\{x_i(\sigma^k)\mid 1\leq i\leq n, 0\leq k < d_i\}$. Moreover, since
$R\subset P$ and $f_i = g_i\cdot x_i(1)$ one has that
$\NF(x_i(\sigma^k)) = \NF(\sigma^k\cdot x_i(1)) = A^k x_i(1)$, for all
$k\geq 0$ and $x_i(1)\in X(1)'$.
\end{proof}

Note that $R$ is $\Sigma$-generated by the set $X(1)' =
\{x_i(1)\mid 1\leq i\leq n, d_i > 0\}$. Since $P$ is a free $\Sigma$-algebra,
a surjective $\Sigma$-algebra homomorphism $\varphi:P\to R$ is defined
such that
\[
x_i(1)\mapsto
\left\{
\begin{array}{cl}
x_i(1) & \mbox{if}\ d_i > 0, \\
0 & \mbox{otherwise}.
\end{array}
\right.
\]
Then, the above result states that the $\Sigma$-ideal $\Ker\varphi\subset P$
of all $\Sigma$-algebra relations satisfied by the generating set $X(1)'\cup \{0\}$
of $R$ is exactly $J$.

Assume now that $n = 1$, that is, $X = \{x\}$ and $\Sigma =
\langle \sigma_1,\ldots,\sigma_r \rangle$. Then $P' = P/J$ where $J =
\langle f_1,\ldots,f_r \rangle_\Sigma$
with $f_j = \sum_{0\leq k\leq d_j} c_{jk} x(\sigma_j^k)$ ($c_{jk}\in K,
c_{jd_j} = 1$). Define $Q = K[\sigma_1,\ldots,\sigma_r]$ the algebra
of polynomials in the variables $\sigma_j$ and denote $g_j =
\sum_{0\leq k\leq d_j} c_{jk} \sigma_j^k\in Q$. One has clearly that
$f_j = g_j\cdot x(1)$. As before, we consider the companion matrix $A_j$
of the polynomial $g_j$ in the single variable $\sigma_j$.
If $d = \prod_j d_j$ then the monoid $\Sigma =
\Sigma_1\times\cdots\times\Sigma_r$ ($\Sigma_j = \langle \sigma_j \rangle$),
that is, the algebra $Q = Q_1 \otimes\cdots\otimes Q_r$ ($Q_j = K[\sigma_j]$)
acts linearly over the space $V = K^d$ by means of the representation
\[
\sigma_1^{k_1} \cdots \sigma_r^{k_r}\mapsto
A_1^{k_1} \otimes \cdots \otimes A_r^{k_r},
\]
where $A_1^{k_1} \otimes \cdots \otimes A_r^{k_r}$ denotes the Kronecker
product of the matrices $A_j^{k_j}$. In other words, the $Q$-module $V$
is the tensor product $V = V_1 \otimes\cdots\otimes V_r$ where $V_j$
is the cyclic $Q_j$-module defined by the representation $\sigma_j^k\mapsto A_j^k$.
If $\{v_{k_1} \otimes\cdots\otimes v_{k_r}\}_{1\leq k_j\leq d_j}$ is the
canonical basis of $V$, we put $x(\sigma_1^{k_1}\cdots\sigma_r^{k_r}) =
v_{k_1+1} \otimes\cdots\otimes v_{k_r+1}$, for all $1\leq j\leq r, 0\leq k_j< d_j$.
One has then
\[
x(\sigma_1^{k_1}\cdots\sigma_r^{k_r}) =
(A_1^{k_1} \otimes\ldots\otimes A_r^{k_r}) x(1) =
(\sigma_1^{k_1}\cdots\sigma_r^{k_r})\cdot x(1),
\]
that is, $V$ is a cyclic module generated by $x(1)$. Denote now by $R$
the polynomial algebra generated by the finite set of variables
$X(\Sigma)' = \{x(\sigma_1^{k_1}\cdots\sigma_r^{k_r})\mid 1\leq j\leq r,
0\leq k_j < d_j\}$, that is, $V$ is the subspace of linear forms of $R$.
Again, we extend the action of the monoid $\Sigma =
\langle \sigma_1,\ldots,\sigma_r \rangle$ to the polynomial algebra $R$
by putting, for all $k_1,\ldots,k_r\geq 0$ and $x(\sigma)\in X(\Sigma)'$
\[
(\sigma_1^{k_1} \cdots \sigma_r^{k_r})\cdot x(\sigma) =
(A_1^{k_1} \otimes \cdots \otimes A_r^{k_r}) x(\sigma).
\]

\begin{proposition}
If $X = \{x\}$ then $P',R$ are $\Sigma$-isomorphic.
\end{proposition}

\begin{proof}
Assume $d\neq 0$, that is, $d_j\neq 0$ for all $j$. Again, by Proposition
\ref{finlinact} one has that the set $\{f_j\}$ is a \Gr\ $\Sigma$-basis
of $J\subset P$ and the set of normal variables modulo $J$ is clearly
$X(\Sigma)' = \{x(\sigma_1^{k_1}\cdots\sigma_r^{k_r})\mid 1\leq j\leq r,
0\leq k_j < d_j\}$. Moreover, because $R\subset P$ and $f_j = g_j\cdot x(1)$
we obtain that, for all $k_1,\ldots,k_r\geq 0$
\[
\NF(x(\sigma_1^{k_1} \cdots \sigma_r^{k_r})) =
\NF((\sigma_1^{k_1} \cdots \sigma_r^{k_r})\cdot x(1)) =
(A_1^{k_1} \otimes\ldots\otimes A_r^{k_r}) x(1).
\]
Finally, if $d = 0$ then $P' = R = K$.
\end{proof}

Note that for $d\neq 0$ one has that $R$ is $\Sigma$-generated
by the element $x(1)$. Then, the above result implies that
the $\Sigma$-ideal $J\subset P$ coincides with the ideal
of $\Sigma$-algebra relations satisfied by the generator $x(1)$,
that is, it is the kernel of the $\Sigma$-algebra epimorphism $P\to R$
such that $x(1)\mapsto x(1)$.

Consider finally the general case for the $\Sigma$-algebra $P' = P/J$
where $J = \langle f_{ij} \rangle_\Sigma$ and $f_{ij} =
\sum_{0\leq k\leq d_{ij}} c_{ijk} x_i(\sigma_j^k)$
with $c_{ijk}\in K, c_{ijd_{ij}} = 1$, for all $1\leq i\leq n$ and
$1\leq j\leq r$. By combining the previous results, one may conclude
that such a structure arises from the $Q$-module $V = K^d$ where
$d = \sum_{1\leq i\leq n} \prod_{1\leq j\leq r} d_{ij}$ and the
representation is given by the mapping
\[
\prod_j \sigma_j^{k_j} \mapsto \bigoplus_i \bigotimes_j A_{ij}^{k_j}
\]
where $A_{ij}$ is the companion matrix of the polynomial
$g_{ij} = \sum_{0\leq k\leq d_{ij}} c_{ijk} \sigma_j^k$.
In other words, we have that $V = \bigoplus_i \bigotimes_j V_{ij}$ where
$V_{ij}$ is the cyclic $Q_j$-module annihilated by
the ideal $\langle g_{ij} \rangle\subset Q_j$. By denoting $x_i(1)$
the generator of the $Q$-module $\bigotimes_j V_{ij}$, we obtain that
$P'$ is isomorphic to the $\Sigma$-algebra $R = K[X(\Sigma)']$ where
$X(\Sigma)' = \{x_i(\sigma_1^{k_{i1}}\cdots\sigma_r^{k_{ir}})\mid
1\leq i\leq n, 1\leq j\leq r, 0\leq k_{ij} < d_{ij}\}$ is the canonical
basis of the space $V$. Then, one has that $J = \langle f_{ij} \rangle_\Sigma$
is exactly the $\Sigma$-ideal of $\Sigma$-algebra relations satisfied
by generating set $X(1)'\cup \{0\}$ of $R$.


\section{Another example}

A long-lasting problem in \Gr\ bases theory is about the possibility
to accord the definition and the computation of such bases to some form
of symmetry, typically defined by groups, which one may have on the generators
or on the ideal itself of some polynomial algebra (see for instance
\cite{BF,Ga}). The main objection against this possibility is that monomial
orderings cannot be defined consistently with the group action which implies
that the symmetry disappears in the \Gr\ basis. In fact, if the symmetry
is defined by a monoid $\Sigma$ isomorphic to $\N^r$ we have found that
the notion of $\Sigma$-ideal perfectly accords with monomial orderings
and \Gr\ bases. Moreover, in the previous section we have shown that by means
of the notion of quotient $\Sigma$-algebra and the corresponding \Gr\ bases tools
one can deal with symmetries defined by suitable finite dimensional commutative
algebras. Among them one finds group algebras of finite abelian groups and
therefore this section is devoted to such a case. In other words, we will show
that \Gr\ bases of ideals having a finite abelian group symmetry can be ``tamed''
by means of $\Sigma$-algebras and their quotients.

We fix now a setting that has been recently considered in \cite{St}. Note that
in our approach all computations can be performed over any field (of constants)
but in \cite{St} the base field is required to contain roots of unity.
Fix $r = 1$, that is, $\Sigma = \langle \sigma \rangle$ and $Q = K[\sigma]$.
Consider $\Sym_d$ the symmetric group on $d$ elements and let $\gamma\in\Sym_d$
be any permutation. Denote $\Gamma = \langle \gamma \rangle \subset\Sym_d$
the cyclic subgroup generated by $\gamma$. Moreover, let $\gamma =
\gamma_1\cdots\gamma_n$ be the cycle decomposition of $\gamma$
and denote by $d_i$ the length of the cycle $\gamma_i$. Consider the polynomial
algebra $R = K[x_i(\sigma^j)\mid 1\leq i\leq n, 0\leq j < d_i]$ and identify
the subset $\{x_i(1),\ldots,x_i(\sigma^{d_i-1})\}$ with the support of
the cycle $\gamma_i$. Define $\Aut_K(R)$ the group of $K$-algebra automorphisms
of $R$. Clearly $R$ is a $\Gamma$-algebra, that is, there is a (faithful) group
representation $\rho':\Gamma\to\Aut_K(R)$. Consider now the polynomials
$g_i = \sigma^{d_i} - 1\in Q$ and define the $d\times d$ block-diagonal matrix
\[
A =
\left(
\begin{array}{cccc}
A_1 &  0  & \ldots & 0 \\
 0  & A_2 & \ldots & 0 \\
\vdots & \vdots & & \vdots \\
 0  &  0  & \ldots & A_n \\
\end{array}
\right)
\]
where each block $A_i$ is the companion matrix of the polynomial $g_i$
which is the permutation matrix
\[
A_i =
\left(
\begin{array}{cccccc}
0 & 0 & \ldots & 0 & 1 \\
1 & 0 & \ldots & 0 & 0 \\
0 & 1 & \ldots & 0 & 0 \\
\vdots & \vdots & \ddots & \vdots & \vdots \\
0 & 0 & \ldots & 1 & 0 \\
\end{array}
\right).
\]
If we order the variables of $R$ as $x_1(1),\ldots,x_1(\sigma^{d_1-1}),\ldots,
x_n(1),\ldots,x_n(\sigma^{d_n-1})$ then the representation $\rho'$ is defined
as $\gamma^k\cdot x_i(\sigma^j) = A^k x_i(\sigma^j)$, for all $i,j,k$.
In other words, by Proposition \ref{iso1} one has that $R$ is a $\Sigma$-algebra
isomorphic to $P' = P/J$ where $J = \langle f_1,\ldots,f_n \rangle_\Sigma$
and $f_i = g_i\cdot x_i(1) = x_i(\sigma^{d_i}) - x_i(1)\in P$. Consider now
a $\Gamma$-ideal (equivalently a $\Sigma$-ideal) $L' =
\langle h_1,\ldots,h_m \rangle_\Gamma\subset R$ and define the $\Sigma$-ideal
$L = \langle h_1,\ldots,h_m, f_1,\ldots,f_n \rangle_\Sigma\subset P$.  Note that
$\Gamma$-ideals are called ``symmetric ideals'' in \cite{St}. According with
Definition \ref{quogb} and the identification of $R$ with the quotient $P'$
one has that $G'\subset L'$ is a \Gr\ $\Gamma$-basis (equivalently $\Sigma$-basis)
of $L'$ if by definition $G'\cup \{f_1,\ldots,f_n\}$ is a \Gr\ $\Sigma$-basis
of $L$. In practice, the computation of $G'$ is obtained by the algorithm
$\SigmaGBasis$ which terminates owing to Corollary \ref{fingb}.

To illustrate the method we fix now $\gamma = (1 2 3 4 5 6 7 8)\in\Sym_8$
and $K = \Q$. To simplify the variables notation we identify $\Sigma$ with $\N$,
that is, $R = K[x(0),x(1),\ldots,x(7)]$. Consider the following $\Gamma$-ideal
of $R$
\begin{equation*}
\begin{gathered}
L' = \langle x(0)x(2) - x(1)^2, x(0)x(3) - x(1)x(2) \rangle_\Gamma = \\
\langle
x(0)x(2) - x(1)^2,
x(1)x(3) - x(2)^2,
x(2)x(4) - x(3)^2,
x(3)x(5) - x(4)^2, \\
x(4)x(6) - x(5)^2,
x(5)x(7) - x(6)^2,
x(7)^2 - x(0)x(6),
x(1)x(7) - x(0)^2, \\
x(0)x(3) - x(1)x(2),
x(1)x(4) - x(2)x(3),
x(2)x(5) - x(3)x(4), \\
x(3)x(6) - x(4)x(5),
x(4)x(7) - x(5)x(6),
x(6)x(7) - x(0)x(5), \\
x(0)x(7) - x(1)x(6),
x(2)x(7) - x(0)x(1)
\rangle.
\end{gathered}
\end{equation*}
Note that $x(0)x(2) - x(1)^2, x(1)x(3) - x(2)^2, x(0)x(3) - x(1)x(2)$
are well-known equations of the twisted cubic in $\P^3$. Define now
$f = x(8) - x(0)\in P$ and hence $R = P' = P/J$ where $J = \langle f \rangle_\Sigma$.
Then, a \Gr\ $\Gamma$-basis (or $\Sigma$-basis) of $L'$ is obtained by
computing a \Gr\ $\Sigma$-basis of the ideal
\[
L = \langle x(0)x(2) - x(1)^2, x(0)x(3) - x(1)x(2), f \rangle_\Sigma\subset P.
\]
Fix for instance the lexicographic monomial ordering on $P$ (hence on $R$)
with $x(0)\prec x(1)\prec\ldots$ which is clearly a $\Sigma$-ordering.
The usual minimal \Gr\ basis of $L'$ consists of 54 elements whose leading
monomials are
\begin{equation*}
\begin{gathered}
x(7)^2,
x(6)x(7), \\
x(0)x(2)\to x(1)x(3)\to x(2)x(4)\to x(3)x(5)\to x(4)x(6)\to x(5)x(7), \\
x(0)x(3)\to x(1)x(4)\to x(2)x(5)\to x(3)x(6)\to x(4)x(7),
x(2)x(7), \\
x(1)x(7),
x(0)x(7),
x(6)^3,
x(0)x(4)^2\to x(1)x(5)^2\to x(2)x(6)^2, \\
x(0)^2x(4)\to x(1)^2x(5)\to x(2)^2x(6)\to x(3)^2x(7),
x(0)^2x(6),
x(0)x(6)^2, \\
x(1)x(6)^2,
x(1)^2x(6),
x(3)^2x(4)\to x(4)^2x(5)\to x(5)^2x(6), \\
x(4)x(5)^2\to x(5)x(6)^2,
x(0)x(1)x(6),
x(0)x(4)x(5)\to x(1)x(5)x(6), \\
x(0)x(5)x(6),
x(1)x(2)x(6),
x(2)^4\to x(3)^4\to x(4)^4\to x(5)^4,
x(0)^3x(5), \\
x(0)x(5)^3,
x(2)^3x(3),
x(2)x(3)^3\to x(3)x(4)^3,
x(0)^2x(5)^2,
x(0)^2x(1)x(5), \\
x(2)^2x(3)^2,
x(1)^2x(2)^3,
x(1)^4x(2)^2,
x(1)^6x(2),
x(1)^8.
\end{gathered}
\end{equation*}
The arrow between two monomials means that a monomial can be obtained
by the previous one by means of the $\Sigma$-action. Then, the minimal \Gr\
$\Gamma$-basis of $L'$ has just 32 elements and their leading monomials are
\begin{equation*}
\begin{gathered}
x(7)^2,
x(6)x(7),
x(0)x(2),
x(0)x(3),
x(2)x(7),
x(1)x(7),
x(0)x(7),
x(6)^3, \\
x(0)x(4)^2,
x(0)^2x(4),
x(0)^2x(6),
x(0)x(6)^2,
x(1)x(6)^2,
x(1)^2x(6),
x(3)^2x(4), \\
x(4)x(5)^2,
x(0)x(1)x(6),
x(0)x(4)x(5),
x(0)x(5)x(6),
x(1)x(2)x(6),
x(2)^4, \\
x(0)^3x(5),
x(0)x(5)^3,
x(2)^3x(3)
x(2)x(3)^3,
x(0)^2x(5)^2,
x(0)^2x(1)x(5), \\
x(2)^2x(3)^2,
x(1)^2x(2)^3,
x(1)^4x(2)^2,
x(1)^6x(2),
x(1)^8.
\end{gathered}
\end{equation*}
In other words, our approach based on $\Sigma$-compatible structures is able
to define appropriately a \Gr\ basis that generates a group invariant
ideal up to the group action and this basis is actually more compact than
the usual \Gr\ basis. The elements of the minimal \Gr\ $\Gamma$-basis of $L'$
are the following ones
\begin{equation*}
\begin{gathered}
x(7)^2 - x(0)x(6),
x(6)x(7) - x(0)x(5),
x(0)x(2) - x(1)^2,
x(0)x(3) - x(1)x(2), \\
x(2)x(7) - x(0)x(1),
x(1)x(7) - x(0)^2,
x(0)x(7) - x(1)x(6),
x(6)^3 - x(0)x(5)^2, \\
x(0)x(4)^2 - x(2)x(3)^2,
x(0)^2x(4) - x(1)^2x(2),
x(0)^2x(6) - x(0)x(1)x(5), \\
x(0)x(6)^2 - x(1)x(5)x(6),
x(1)x(6)^2 - x(0)^2x(5),
x(1)^2x(6) - x(0)^3, \\
x(3)^2x(4) - x(0)x(1)^2,
x(4)x(5)^2 - x(0)x(1)x(5),
x(0)x(1)x(6) - x(2)^2x(3), \\
x(0)x(4)x(5) - x(0)^2x(1),
x(0)x(5)x(6) - x(3)x(4)^2, \\
x(1)x(2)x(6) - x(0)x(4)x(5),
x(2)^4 - x(0)^4,
x(0)^3x(5) - x(3)^3x(4), \\
x(0)x(5)^3 - x(3)x(4)^3,
x(2)^3x(3) - x(0)^3x(1),
x(2)x(3)^3 - x(0)x(1)^3, \\
x(0)^2x(5)^2 - x(2)^2x(3)^2,
x(0)^2x(1)x(5) - x(3)^2x(4)^2,
x(2)^2x(3)^2 - x(0)^2x(1)^2, \\
x(1)^2x(2)^3 - x(0)^5,
x(1)^4x(2)^2 - x(0)^6,
x(1)^6x(2) - x(0)^7,
x(1)^8 - x(0)^8.
\end{gathered}
\end{equation*}
We have computed these elements by applying the algorithm \SigmaGBasis\
to the $\Sigma$-ideal $L\subset P$ in the same way as for the example
in Section 5. For details about different strategies to implement this method
we refer to \cite{LS}.

\section{Conclusions and further directions}

In this paper we showed that a viable theory of \Gr\ bases exists for the algebra
of partial difference polynomials which implies that one can perform symbolic (formal)
computations for systems of partial difference equations. In fact, we prove that
such \Gr\ bases can be computed in a finite number of steps when truncated
with respect to an appropriate grading or when they contain elements
with suitable linear leading monomials. Precisely, since the algebras
of difference polynomials are free objects in the category of $\Sigma$-algebras
where $\Sigma$ is a monoid isomorphic to $\N^r$, we obtained the latter result
as a Noetherianity criterion for a class of finitely generated $\Sigma$-algebras.
Among such Noetherian $\Sigma$-algebras one finds polynomial algebras
in a finite number of variables where a tensor product of a finite number
of algebras generated by single matrices acts over the subspace of linear forms.
Considering that such commutative tensor algebras include group algebras
of finite abelian groups one obtains that there exists a consistent
\Gr\ basis theory for ideals of finitely generated polynomial algebras
that are invariant under such groups. In our opinion, this represents
an interesting step in the direction of development of computational methods
for ideals or algebras that are subject to group or algebra symmetries.

As for further developments, we may suggest that the study of important
structures related to \Gr\ bases like Hilbert series and free resolutions
should be developed in the perspective that their definition and computation
has to be consistent to the symmetry one defines eventually on a polynomial
algebra. An important work in this direction is contained in \cite{KLMP}.
Finally, the problem of studying conditions providing $\Sigma$-Noetherianity
(instead of simple Noetherianity) for finitely generated $\Sigma$-algebras 
is also an intriguing subject.


\section*{Acknowledgments}

The authors would like to express their gratitude to the reviewers
for all valuable remarks that have helped to make the paper more readable.


\end{document}